\documentclass[reqno]{amsart}
\usepackage{amssymb,eucal,latexsym, mathrsfs,enumerate,graphicx,url,tikz,tikz-cd}

\newcommand{\mycircle}[1]
{\draw[fill=white, line width=1pt] #1 circle[radius=1.2mm]}

\tikzset{every picture/.style={line width=0.7pt}}

{\end{enumerate}}
{\end{enumerate}}
\newenvironment{enumerater}{\begin{enumerate}[\upshape (1)]}%
{\end{enumerate}}

\newcommand{\oo}[1]{\left]{#1}\right[}
\newcommand{\oc}[1]{\left]{#1}\right]}
\newcommand{\co}[1]{\left[{#1}\right[}

\newcommand{\Pow}{\mathfrak{P}}
\newcommand{\pup}[1]{\textup{(}{#1}\textup{)}}

\newcommand{\lgrp}{$\ell$-group}
\newcommand{\lsgrp}{$\ell$-sub\-group}
\newcommand{\lhom}{$\ell$-ho\-mo\-mor\-phism}
\newcommand{\lemb}{$\ell$-em\-bed\-ding}

\newcommand{\jirr}{join-ir\-re\-duc\-i\-ble}

\newcommand{\jh}{join-ho\-mo\-mor\-phism}
\newcommand{\mh}{meet-ho\-mo\-mor\-phism}
\newcommand{\eqdef}{\overset{\mathrm{def}}{=}}

\newcommand{\rF}{\mathrm{F}}
\newcommand{\xF}{\mathbf{F}}

\DeclareMathOperator{\At}{At}

\newcommand{\OK}{\ol{\Bbbk}^{\,+}}
\newcommand{\OQ}{\ol{\QQ}^{\,+}}

\newcommand{\ga}{\alpha}

\newcommand{\gc}{\gamma}
\newcommand{\gd}{\delta}
\newcommand{\gf}{\varphi}

\newcommand{\gl}{\lambda}

\newcommand{\gs}{\sigma}
\newcommand{\go}{\omega}

\newcommand{\bga}{\boldsymbol{\alpha}}

\newcommand{\bgd}{\boldsymbol{\delta}}

\newcommand{\bck}[1]{[\![{#1}]\!]}

\newcommand{\sd}{\mathbin{\smallsetminus}}

\newcommand{\jz}{$(\vee,0)$}
\newcommand{\js}{join-semi\-lat\-tice}

\newcommand{\jzs}{\jz-semi\-lat\-tice}

\newcommand{\ajs}{al\-most join-sem\-i\-lat\-tice}
\newcommand{\pjs}{pseu\-do join-sem\-i\-lat\-tice}

\newcommand{\two}{\mathbf{2}}

\newcommand{\ol}[1]{\overline{#1}}

\newcommand{\pI}[1]{\bigl({#1}\bigr)}

\newcommand{\set}[1]{\left\{#1\right\}}
\newcommand{\setm}[2]{\set{{#1}\mid{#2}}}
\newcommand{\vecm}[2]{({#1}\mid{#2})}
\newcommand{\Vecm}[2]{\left({#1}\mid{#2}\right)}
\newcommand{\seq}[1]{\langle{#1}\rangle}
\newcommand{\seql}[1]{{\langle{#1}\rangle}^{\ell}}

\newcommand{\id}{\mathrm{id}}
\newcommand{\dnw}{\mathbin{\downarrow}}

\newcommand{\upw}{\mathbin{\uparrow}}

\newcommand{\Sor}{\mathbin{\bigtriangledown}}
\newcommand{\es}{\varnothing}
\newcommand{\res}{\mathbin{\restriction}}

\newcommand{\ZZ}{\mathbb{Z}}
\newcommand{\QQ}{\mathbb{Q}}

\newcommand{\Bool}{\mathbf{Bool}}

\DeclareMathOperator{\Id}{Id}
\DeclareMathOperator{\Idc}{Id_c}
\DeclareMathOperator{\Cs}{Cs}
\DeclareMathOperator{\Csc}{Cs_c}

\newcommand{\cA}{{\mathcal{A}}}

\newcommand{\cO}{{\mathcal{O}}}
\newcommand{\cP}{{\mathcal{P}}}

\newcommand{\cS}{{\mathcal{S}}}

\newcommand{\cV}{{\mathcal{V}}}

\numberwithin{equation}{section}

\newtheorem*{stat}{\name}
\newcommand{\name}{testing}

\theoremstyle{plain}

\newtheorem{theorem}{Theorem}[section]
\newtheorem{proposition}[theorem]{Proposition}
\newtheorem{corollary}[theorem]{Corollary}
\newtheorem{lemma}[theorem]{Lemma}

\newtheorem{claim}{Claim}

\theoremstyle{definition}

\newtheorem{definition}[theorem]{Definition}

\newtheorem*{problem}{Problem}

\theoremstyle{remark}
\newtheorem{remark}[theorem]{Remark}

\newcommand{\qedc}{{\qed}~{\rm Claim~{\theclaim}.}}
\newcommand{\qedsc}{{\qed}~{\rm Claim.}}

\newenvironment{cproof}
{\begin{proof}[Proof of Claim.]}
{\qedc\renewcommand{\qed}{}\end{proof}}

\numberwithin{figure}{section}
\numberwithin{table}{section}

\newcommand{\ba}{\boldsymbol{a}}
\newcommand{\bb}{\boldsymbol{b}}
\newcommand{\bc}{\boldsymbol{c}}

\newcommand{\bu}{\boldsymbol{u}}

\newcommand{\bx}{\boldsymbol{x}}
\newcommand{\by}{\boldsymbol{y}}
\newcommand{\bz}{\boldsymbol{z}}

\newcommand{\bV}{\boldsymbol{V}}

\newcommand{\bA}{\boldsymbol{A}}
\newcommand{\bB}{\boldsymbol{B}}

\newcommand{\bD}{\boldsymbol{D}}

\newcommand{\bX}{\boldsymbol{X}}

\newcommand{\vx}{\mathsf{x}}
\newcommand{\vy}{\mathsf{y}}
\newcommand{\vz}{\mathsf{z}}

\newcommand{\scL}{\mathbin{\mathscr{L}}}

\title[Cevian operations]%
{Cevian operations on distributive lattices}

\author[F. Wehrung]{Friedrich Wehrung}
\address{LMNO, CNRS UMR 6139\\
D\'epartement de Math\'ematiques\\
Universit\'e de Caen Normandie\\
14032 Caen cedex\\
France}
\email{friedrich.wehrung01@unicaen.fr}
\urladdr{https://wehrungf.users.lmno.cnrs.fr}

\date{\today}

\subjclass[2010]{06D05; 06D35; 06F20; 03E02; 03E05; 18C35; 18A25; 18A30; 18A35; 18B35; 46A55}

\keywords{Cevian; completely normal; lattice-ordered; Abelian; group; convex; ideal; countable; distributive; lattice; condensate; colimit; finitely presented}

\begin{document}

\begin{abstract}
We construct a completely normal bounded distributive lattice~$D$ in which for every pair $(a,b)$ of elements, the set
$\setm{x\in D}{a\leq b\vee x}$ has a countable coinitial subset, such that~$D$ does not carry any binary operation~$\sd$ satisfying the identities $\vx\leq\vy\vee(\vx\sd\vy)$, $(\vx\sd\vy)\wedge(\vy\sd\vx)=0$, and $\vx\sd\vz\leq(\vx\sd\vy)\vee(\vy\sd\vz)$.
In particular, $D$ is not a homomorphic image of the lattice of all finitely generated convex $\ell$-subgroups of any (not necessarily Abelian) \lgrp.
It has~$\aleph_2$ elements.
This solves negatively a few problems stated by Iberkleid,  Mart{\'{\i}}nez, and McGovern in 2011 and recently by the author.
This work also serves as preparation for a forthcoming paper in which we prove that \emph{for any infinite cardinal~$\gl$, the class of Stone duals of spectra of all Abelian \lgrp{s} with order-unit is not closed under $\scL_{\infty\gl}$-elementary equivalence}.
\end{abstract}

\maketitle


\section{Introduction}\label{S:Intro}

It has been known since the seventies that for any Abelian lattice-ordered group (from now on \emph{\lgrp})~$G$, the distributive lattice~$\Idc{G}$ of all finitely generated (equivalently principal) $\ell$-ideals of~$G$ is \emph{completely normal}, that is, it satisfies the statement
 \[
 (\forall a,b)(\exists x,y)(a\vee b=a\vee y=x\vee b\text{ and }
 x\wedge y=0)\,.
 \]
Delzell and Madden found in~\cite{DelMad1994} an example of a completely normal bounded distributive lattice which is not isomorphic to~$\Idc{G}$ for any Abelian \lgrp~$G$.
Since then, the problem of characterizing all lattices of the form~$\Idc{G}$ has been widely open, possibly under various equivalent forms, one of which being the \emph{MV-spectrum problem} (cf. Mun\-di\-ci \cite[Problem~2]{Mund2011}).
The author's paper~\cite{MV1} settles the \emph{countable} case, by proving that complete normality is then sufficient.
However, moving to the uncountable case, we prove in~\cite{MV1} that \emph{the class of all lattices of the form~$\Idc{G}$, for Abelian \lgrp{s}~$G$ with order-unit, is not closed under $\scL_{\infty\go}$-elementary equivalence}.

A remarkable additional property of lattices of the form~$\Idc{G}$, for Abelian \lgrp{s}~$G$, was coined, under different names, on the one hand in Cignoli \emph{et al.}~\cite{CGL}, where it was denoted by~$(\Id\go)$, on the other hand in Iberkleid \emph{et al.}~\cite{IMM2011}, where it was called ``$\gs$-Conrad''.
In~\cite{MV1} we express that property by an $\scL_{\go_1\go_1}$ sentence of lattice theory that we call \emph{having countably based differences} (cf. Subsection~\ref{Su:Posets}).
This property is trivially satisfied in the countable case, but fails for various uncountable examples such as Delzell and Madden's.

In this paper we prove (cf. Theorem~\ref{T:Main}) that requiring countably based differences, together with complete normality, is not sufficient to characterize distributive lattices of the form~$\Idc{G}$ for Abelian \lgrp{s}~$G$.
It turns out that our counterexample also gives a strong negative answer to \cite[Question~4.3.1]{IMM2011}, by proving that ``$\gs$-Conrad does not imply Conrad'' (it was proved in~\cite{IMM2011} that normal-valued Conrad implies $\gs$-Conrad).
It also proves that the implication (4)$\Rightarrow$(5), in \cite[\S~4]{IMM2011}, is strict (\cite{IMM2011} achieved a partial result in that direction).
Our main counterexample has cardinality~$\aleph_2$\,.

The proof of our main result is achieved in several steps.
We observe (cf. Proposition~\ref{P:CevaSclgrp}) that for any (not necessarily Abelian) \lgrp~$G$, the (completely normal, distributive) lattice~$\Csc{G}$ of all finitely generated convex $\ell$-subgroups of~$G$ carries a binary operation~$\sd$ satisfying the identities
$\vx\leq\vy\vee(\vx\sd\vy)$, $(\vx\sd\vy)\wedge(\vy\sd\vx)=0$, and $\vx\sd\vz\leq(\vx\sd\vy)\vee(\vy\sd\vz)$.
We call such operations \emph{Cevian operations} and we call such lattices \emph{Cevian lattices} (Definition~\ref{D:CevaDiff}).

We thus need to construct a non-Cevian completely normal distributive lattice with zero and countably based differences.
In order to achieve this, we first solve the problem at \emph{diagram level}, by constructing (cf. Lemma~\ref{L:IdcANonRepr}) a $\set{0,1}^3$-indexed commutative diagram, of countable completely normal distributive lattices with zero, which is a counterexample to a diagram analogue of a ``local'' form of the main question.
This diagram is obtained by applying the functor~$\Idc$ to a certain \emph{non-com\-mu\-ta\-tive}%
\footnote{We need the non-commutativity of the diagram~$\vec{A}$---otherwise, by definition, $\Idc\vec{A}$ would not be a diagram counterexample!
This will be strongly illustrated in Proposition~\ref{P:CevaSclgrp}.}
diagram of Abelian \lgrp{s}, which we denote by~$\vec{A}$ (cf. Section~\ref{S:vecA}).

The proof of Lemma~\ref{L:IdcANonRepr} rests on a lattice-theoretical interpretation, established in Proposition~\ref{P:LatCeva}, of the configuration underlying Ceva's Theorem in elementary plane geometry.

Our final line of argument relies on the results of the monograph Gillibert and Wehrung~\cite{Larder}, which sets up a machinery making it possible to turn certain \emph{diagram counterexamples} to \emph{object counterexamples}, \emph{via} constructs called \emph{condensates}, infinite combinatorial objects called \emph{lifters}, and a technical result called the \emph{Armature Lemma}.
We summarize the required machinery in Section~\ref{S:CrashCondens}, and we embark on our main result's final proof in Section~\ref{S:aleph2}.

Since, as mentioned above, having countably based differences is an~$\scL_{\go_1\go_1}$ sentence, this is related to the question, stated as \cite[Problem~1]{MV1}, whether there exists an infinite cardinal~$\gl$ such that the class of all lattices of the form~$\Idc{G}$, for Abelian \lgrp{s}~$G$ (equivalenly, Stone dual lattices of spectra of Abelian \lgrp{s}), can be characterized by some class of~$\scL_{\infty\gl}$ sentences.
In our subsequent paper~\cite{NonElt} we prove that this is not so, by proving the even stronger result that \emph{the class of Stone duals of spectra of all Abelian \lgrp{s} with order-unit is not closed under $\scL_{\infty\gl}$-elementary equivalence}.

\section{Notation, terminology, and basic concepts}\label{S:Basic}

\subsection{Sets, posets}\label{Su:Posets}
Following standard set-theoretical notation, we denote by~$\go$ the first infinite ordinal and also the set of all nonnegative integers.
For a natural number~$n$, we denote by~$\aleph_n$ the $n$th transfinite cardinal number, also denoted by~$\go_n$ in case it should be viewed as an ordinal.
The set of all finite subsets of a set~$X$ will be denoted by~$[X]^{<\go}$.
We denote by~$\Pow(X)$, or just~$\Pow{X}$, the powerset of a set~$X$.
By ``countable'' we will always mean ``at most countable''.

For any element~$a$ in a poset (i.e., partially ordered set)~$P$, we set
 \[
 P\dnw a\eqdef\setm{p\in P}{p\leq a}\,,\quad
 P\upw a\eqdef\setm{p\in P}{p\geq a}\,.
 \]
A subset~$X$ of~$P$ is
\begin{itemize}
\item[---]
a \emph{lower subset} (resp., \emph{upper subset}) of~$P$ if $P\dnw x\subseteq X$ (resp., $P\upw x\subseteq X$) whenever $x\in X$;

\item[---]
an \emph{ideal} of~$P$ if it is a nonempty, upward directed lower subset of~$P$;

\item[---]
\emph{coinitial} in~$P$ if
$P=\bigcup\vecm{P\upw x}{x\in X}$\,.
\end{itemize}
For posets~$P$ and~$Q$, a map $f\colon P\to Q$ is \emph{isotone} if $x\leq y$ implies that $f(x)\leq f(y)$ whenever $x,y\in P$.
We let $\two\eqdef\set{0,1}$, ordered by $0<1$.

We refer to Gr\"atzer~\cite{LTF} for standard facts on lattice theory.
A distributive lattice~$D$ with zero is \emph{completely normal} if for all $x,y\in D$ there are $u,v\in D$ such that $x\leq y\vee u$, $y\leq x\vee v$, and $u\wedge v=0$.
Equivalently (replacing~$u$ by $u\wedge x$ and~$v$ by $v\wedge y$), $x\vee y=x\vee v=u\vee y$ and $u\wedge v=0$.
By a result from Monteiro~\cite{Mont1954}, this is equivalent to saying that the specialization order in the Stone dual of~$D$ is a root system (see also Cignoli \emph{et al.}~\cite{CGL}).

For any elements~$a$ and~$b$ in a \js~$S$, we set, following the notation in the author's paper~\cite{MV1},
 \begin{equation}\label{Eq:aominusb}
 a\ominus_Sb\eqdef\setm{x\in S}{a\leq b\vee x}\,.
 \end{equation}
Following~\cite{MV1}, we say that~$S$ has \emph{countably based differences} if $a\ominus_Sb$ has a countable coinitial subset whenever $a,b\in S$.

Following~\cite{MV1}, we define a \jh~$f\colon A\to B$, between \js{s}, to be \emph{closed} if for all $a,a'\in A$ and~$b\in B$, if $b\in f(a)\ominus_Bf(a')$, there exists $x\in a\ominus_Aa'$ such that $f(x)\leq b$.
In particular, if~$X$ is a coinitial subset of~$a\ominus_Aa'$, then $f[X]$ is a coinitial subset of $f(a)\ominus_Bf(a')$.
We thus get the following lemma.

\begin{lemma}\label{L:Closed2CBD}
Let~$A$ and~$B$ be \js{s} and let $f\colon A\to B$ be a closed \jh.
For all $a,a'\in A$, if $a\ominus_Aa'$ has a countable coinitial subset, then $f(a)\ominus_Bf(a')$ has a countable coinitial subset.
\end{lemma}


For \lgrp{s} we refer to Bigard \emph{et al.} \cite{BKW}, Anderson and Feil~\cite{AnFe}.
All our \lgrp{s} will be written additively (even in the non-com\-mu\-ta\-tive case), with the lattice operations~$\wedge$ and~$\vee$ being given higher precedence than the group operations (e.g., $u+x\wedge y-v=u+(x\wedge y)-v$).
For any \lgrp~$G$, the lattice~$\Cs{G}$ of all convex $\ell$-subgroups of~$G$ is a distributive algebraic lattice, of which the collection~$\Csc{G}$ of all finitely generated convex $\ell$-subgroups is a sublattice; moreover, $\Csc{G}$ is completely normal.
The elements of~$\Csc{G}$ are exactly those of the form
 \[
 \seq{x}_G\eqdef\setm{y\in G}{(\exists n<\go)(|y|\leq n|x|)}\,,
 \quad\text{for }x\in G\ (\text{equivalently, for }x\in G^+)\,.
 \]
We refer the reader to Iberkleid \emph{et al.} \cite[\S~1.2]{IMM2011} for a more detailed overview of the matter.

The lattice~$\Id{G}$ of all $\ell$-ideals (i.e., normal convex $\ell$-subgroups) of~$G$ is a distributive algebraic lattice, isomorphic to the congruence lattice of~$G$.
The \jzs~$\Idc{G}$ of all finitely generated $\ell$-ideals of~$G$ may not be a lattice (cf. Remark~\ref{Rk:CevaSclgrp} for further explanation).
Its elements are exactly those of the form
 \begin{multline*}
 \seql{x}_G\eqdef
 \{y\in G\mid
 \text{there are }n<\go\text{ and conjugates }x_1\,,\dots,x_n\text{ of }|x|\text{ such that}\\
 |y|\leq x_1+\cdots+x_n\}\,,
 \quad\text{for }x\in G\ (\text{equivalently, for }x\in G^+)\,.
 \end{multline*}
As observed in \cite[Subsection~2.2]{MV1}, the assignment~$\Idc$ naturally extends to a \emph{functor} from Abelian \lgrp{s} and \lhom{s} to completely normal distributive lattices with zero and closed $0$-lattice homomorphisms.
In a similar manner, the assignment~$\Csc$ naturally extends to a \emph{functor} from \lgrp{s} and \lhom{s} to completely normal distributive lattices with zero and closed $0$-lattice homomorphisms.
Of course, if~$G$ is Abelian, then $\Cs{G}=\Id{G}$, $\seq{x}_G=\seql{x}_G$\,, and so on.

For any \lgrp~$G$ and any $x,y\in G^+$, let $x\propto y$ hold if $x\leq ny$ for some positive integer~$n$, and let $x\asymp y$ hold if $x\propto y$ and $y\propto x$.

\subsection{Open polyhedral cones}\label{Su:OpPolCones}
Throughout the paper we will denote by~$\QQ$ the ordered field of all rational numbers and by~$\QQ^+$ its positive cone.
For every positive integer~$n$ and every $n$-ary term~$t$ in the similarity type $(0,+,-,\vee,\wedge)$ of \lgrp{s} (in short \emph{$\ell$-term}), we set
 \[
 \bck{t(\vx_1,\dots,\vx_n)\neq0}_n\eqdef
 \setm{(x_1,\dots,x_n)\in(\QQ^+)^n}{t(x_1,\dots,x_n)\neq0}\,,
 \]
and similarly for $\bck{t(\vx_1,\dots,\vx_n)>0}_n$\,.
In particular, for every positive integer~$n$ and all rational numbers~$\gl_1$\,,  \dots, $\gl_n$\,, we get
 \[
 \bck{\gl_1\vx_1+\cdots+\gl_n\vx_n>0}_n=
 \setm{(x_1,\dots,x_n)\in(\QQ^+)^n}
 {\gl_1x_1+\cdots+\gl_nx_n>0}\,;
 \]
we will call such sets \emph{open half-spaces}%
\footnote{This will include ``degenerate'' cases such as the one where all~$\gl_i$ are zero (resp., positive) and should not cause any problem in the sequel.}
of $(\QQ^+)^n$.
Define a \emph{basic open polyhedral cone} of~$(\QQ^+)^n$ as the intersection of a finite, nonempty collection of open half-spaces, and define a \emph{strict open polyhedral cone} of~$(\QQ^+)^n$ as a finite union of basic open polyhedral cones.
Observe that no strict open polyhedral cone of~$(\QQ^+)^n$ contains~$0$ as an element.
For $n\geq2$, the lattice~$\cO_n$ of all strict open polyhedral cones of~$(\QQ^+)^n$ is a bounded distributive lattice, with zero the empty set and with unit $(\QQ^+)^n\setminus\set{0}$.

\subsection{Non-commutative diagrams}\label{Su:NonCommDiag}
Several sections in the paper will involve the concept of a ``non-com\-mu\-ta\-tive diagram''.
A (\emph{commutative}) \emph{diagram}, in a category~$\cS$, is often defined as a functor~$D$ from a category~$\cP$ (the ``indexing category'' of the diagram) to~$\cS$.
Allowing any morphism in~$\cP$ to be sent to more than one morphism in~$\cS$, we get~$D$ as a kind of ``non-deterministic functor''.
Specializing to the case where~$\cP$ is the category naturally assigned to a poset~$P$, we get the following definition.

\begin{definition}\label{D:DiagramP}
Let~$P$ be a poset and let~$\cS$ be a category.
A \emph{$P$-indexed diagram in~$\cS$} is an assignment~$D$, sending each element~$p$ of~$P$ to an object~$D(p)$ (or~$D_p$) of~$\cS$ and each pair $(p,q)$ of elements of~$P$, with $p\leq q$, to a \emph{nonempty set} $D(p,q)$ of morphisms from~$D(p)$ to~$D(q)$\,, such that
\begin{enumerater}
\item $\id_{D(p)}\in D(p,p)$ for every $p\in P$,

\item Whenever $p\leq q\leq r$, $u\in D(p,q)$, and $v\in D(q,r)$, $v\circ u$ belongs to $D(p,r)$.
\end{enumerater}
We say that~$D$ is a \emph{commutative diagram} if each~$D(p,q)$, for~$p\leq q$ in~$P$, is a singleton.
\end{definition}

We will often write poset-indexed commutative diagrams in the form
 \[
 \vec{D}=\vecm{D_p,\gd_p^q}{p\leq q\text{ in }P}\,,
 \]
where all~$D_p$ are objects and all $\gd_p^q\colon D_p\to D_q$ are morphisms subjected to the usual commutation relations (i.e., $\gd_p^p=\id_{D_p}$, $\gd_p^r=\gd_q^r\circ\gd_p^q$ whenever $p\leq q\leq r$); hence $\vec{D}(p,q)=\set{\gd_p^q}$.
If~$P$ is a directed poset we will say that~$\vec{D}$ is a \emph{direct system}.

The following construction will be briefly mentioned in Proposition~\ref{P:vecAIdcComm}, which will play a prominent role in our forthcoming paper~\cite{NonElt}.

\begin{definition}\label{D:D^I}
Let~$I$ be a set, let~$\cS$ be a category with all $I$-indexed products, let~$P$ be a poset, and let~$D$ be a $P$-indexed diagram in~$\cS$.
Denoting by~$P^I$ the $I$-th cartesian power of the poset~$P$, we define a $P^I$-indexed diagram~$D^I$ in~$\cS$ by setting
\begin{enumerater}
\item
$D^I{\vecm{p_i}{i\in I}}\eqdef\prod_{i\in I}D(p_i)$;

\item
whenever $p=\vecm{p_i}{i\in I}$ and $q=\vecm{q_i}{i\in I}$ in~$P$ with $p\leq q$, $D^I(p,q)$ consists of all morphisms of the form $\prod_{i\in I}f_i$ where each $f_i\in D(p_i\,,q_i)$.
\end{enumerater}
\end{definition}

\section{A lattice-theoretical version of Ceva's Theorem}
\label{S:LatCeva}
The goal of this section is to establish Proposition~\ref{P:LatCeva}.
This result solves a problem, mostly of lattice-theoretical nature, on open polyhedral cones in dimension three; its proof involves the main configuration underlying Ceva's Theorem in elementary plane geometry.

Although we will only need to apply Proposition~\ref{P:LatCeva} to the ordered field~$\QQ$ of all rational numbers, it does not bring any additional complexity to state it over an arbitrary totally ordered division ring~$\Bbbk$.
For such a ring, we set
 \begin{align*}
 \Bbbk^+&\eqdef\setm{x\in\Bbbk}{x\geq0}\,,\\
 \Bbbk^{++}&\eqdef\setm{x\in\Bbbk}{x>0}\,,\\
 \OK&\eqdef\Bbbk^+\cup\set{\infty}\,,
 \quad\text{where we declare that }x<\infty
 \text{ whenever }x\in\Bbbk^+\,.
 \end{align*}
For all $x,y\in\OK$, we write
 \[
 [x,y]\eqdef\setm{t\in\OK}{x\leq t\leq y}\,,\quad
 \co{x,y}\eqdef\setm{t\in\OK}{x\leq t<y}\,,
 \]
and so on.
We denote by~$\cO(\OK)$ the set of all finite unions of intervals of~$\OK$ of one of the forms $\co{0,x}$\,, $\oc{y,\infty}$, or $\oo{x,y}$ with $x,y\in\OK$.
For a nonzero pair $(x,y)$ of elements of~$\Bbbk^+$, the expression $x^{-1}y$ is given its usual meaning in~$\Bbbk$ if $x>0$, and extended to the case where $x=0$ (thus $y>0$) by setting $0^{-1}y=\infty$.

\begin{proposition}\label{P:LatCeva}
Let~$\Bbbk$ be a totally ordered division ring.
For all integers $i,j$ with $1\leq i<j\leq 3$, let $U_{ij}\in\cO(\OK)$ and set
 \[
 C_{ij}\eqdef\setm{(x_1,x_2,x_3)\in(\Bbbk^+)^3}
 {(x_i\,,x_j)\neq(0,0)\text{ and }x_i^{-1}x_j\in U_{ij}}\,.
 \]
Suppose that the following statements hold:
\begin{enumerater}
\item\label{0inU1223}
$0\in U_{12}\cap U_{23}\cap U_{13}$\,;

\item\label{U1223notfull}
$\co{0,\infty}\not\subseteq U_{12}$ and $\co{0,\infty}\not\subseteq U_{23}$\,;

\item\label{CevaIneq}
$C_{12}\cap C_{23}\subseteq C_{13}\subseteq C_{12}\cup C_{23}$\,.
\end{enumerater}
Then there are $x,y\in\Bbbk^{++}$ such that $U_{12}=\co{0,x}$\,, $U_{23}=\co{0,y}$\,, and $U_{13}=\co{0,xy}$\,.
\end{proposition}

The conclusion of Proposition~\ref{P:LatCeva} is represented in Figures~\ref{Fig:Ceva1} and~\ref{Fig:CevaCU}.
The configuration represented in Figure~\ref{Fig:Ceva1} will be called a \emph{Ceva configuration} [for the sets~$C_{ij}$].
The sets~$C_{ij}$ are emphasized with a gray shade in Figure~\ref{Fig:CevaCU}.
The sets~$U_{ij}$ are marked in thick black lines, on the boundary of the main triangle, on both pictures.

In all the figures involved in Section~\ref{S:LatCeva}, open polyhedral cones of~$(\Bbbk^+)^3$ will be represented by their intersection with the $2$-simplex
 \[
 \setm{(x_1,x_2,x_3)\in(\Bbbk^+)^3}{x_1+x_2+x_3=1}\,,
 \]
and points will be represented by their \emph{homogeneous coordinates}, so
 \[
 \seq{x,y,z}=\setm{(\gl x,\gl y,\gl z)}{\gl\in\Bbbk\setminus\set{0}}\,.
 \]

\begin{figure}[htb]
{\centering
\begin{tikzpicture}
\draw (0,0)--(4,0)--(2,4)--(0,0);
\draw (2,0)--(2,4);
\draw (0,0)--(3,2);
\draw (1,2)--(4,0);
\draw[line width=2pt] (0,0)--(2,0);
\draw[line width=2pt] (4,0)--(3,2);
\draw[line width=2pt] (1,2)--(0,0);
\mycircle{(0,0)};
\mycircle{(2,0)};
\mycircle{(4,0)};
\mycircle{(3,2)};
\mycircle{(1,2)};
\mycircle{(2,4/3)};
\node at (-.8,0) {$\seq{1,0,0}$};
\node at (4.8,0) {$\seq{0,1,0}$};
\node[above] at (2,4.1) {$\seq{0,0,1}$};
\mycircle{(2,4)};
\node[below] at (2.3,-.1) {$\seq{1,x,0}$};
\node at (1,-.3) {$U_{12}$};
\node at (3.8,2) {$\seq{0,1,y}$};
\node at (4,1) {$U_{23}$};
\node at (.1,2) {$\seq{1,0,xy}$};
\node at (0,1) {$U_{13}$};
\end{tikzpicture}
}
\caption{A Ceva configuration}\label{Fig:Ceva1}
\end{figure}
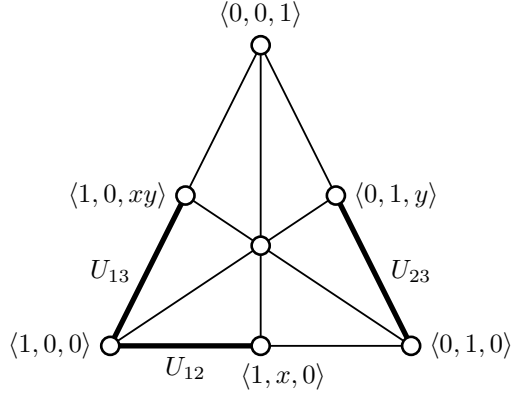

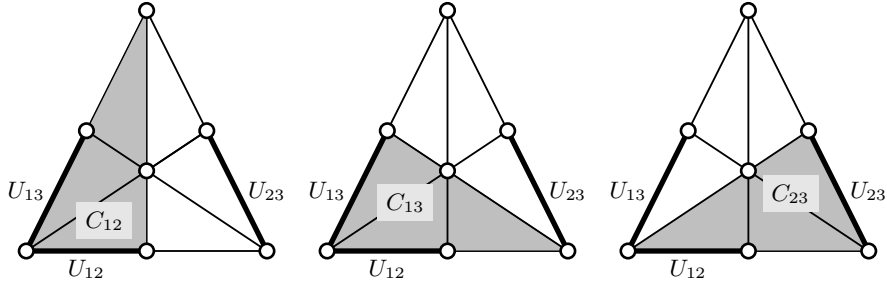
\begin{figure}[htb]
{\centering
\begin{tikzpicture}[scale=.8]
\draw (0,0)--(4,0)--(2,4)--(0,0);
\draw (2,0)--(2,4);
\draw (0,0)--(3,2);
\draw[fill=lightgray,line width=0pt] (0,0)--(2,0)--(2,4)--(0,0);
\draw[line width=2pt] (4,0)--(3,2);
\draw[line width=2pt] (1,2)--(0,0);
\draw[line width=2pt] (0,0)--(2,0);
\mycircle{(2,0)};
\mycircle{(2,4)};
\node at (1,-.3) {\small$U_{12}$};
\node at (4,1) {\small$U_{23}$};
\node at (0,1) {\small$U_{13}$};

\draw (0,0)--(3,2);
\draw (1,2)--(4,0);
\mycircle{(0,0)};
\mycircle{(4,0)};
\mycircle{(2,4/3)};
\mycircle{(1,2)};
\mycircle{(3,2)};

\draw (5,0)--(9,0)--(7,4)--(5,0);
\draw[fill=lightgray,line width=0pt] (5,0)--(9,0)--(6,2);
\draw (7,0)--(7,4);
\draw (5,0)--(8,2);
\draw (6,2)--(9,0);
\draw[line width=2pt] (5,0)--(7,0);
\draw[line width=2pt] (9,0)--(8,2);
\draw[line width=2pt] (6,2)--(5,0);
\mycircle{(5,0)};
\mycircle{(7,0)};
\mycircle{(9,0)};
\mycircle{(8,2)};
\mycircle{(6,2)};
\mycircle{(7,4/3)};
\mycircle{(7,4)};
\node at (6,-.3) {\small$U_{12}$};
\node at (9,1) {\small$U_{23}$};
\node at (5,1) {\small$U_{13}$};

\draw (10,0)--(14,0)--(12,4)--(10,0);
\draw[fill=lightgray,line width=0pt] (10,0)--(14,0)--(13,2);
\draw (12,0)--(12,4);
\draw (10,0)--(13,2);
\draw (11,2)--(14,0);
\draw[line width=2pt] (10,0)--(12,0);
\draw[line width=2pt] (14,0)--(13,2);
\draw[line width=2pt] (11,2)--(10,0);
\mycircle{(10,0)};
\mycircle{(12,0)};
\mycircle{(14,0)};
\mycircle{(13,2)};
\mycircle{(11,2)};
\mycircle{(12,4/3)};
\mycircle{(12,4)};
\node at (11,-.3) {\small$U_{12}$};
\node at (14,1) {\small$U_{23}$};
\node at (10,1) {\small$U_{13}$};

\node[fill=gray!20] at (1.3,.5) {\small$C_{12}$};
\node[fill=gray!20] at (6.3,.8) {\small$C_{13}$};
\node[fill=gray!20] at (12.7,.95) {\small$C_{23}$};
\end{tikzpicture}
}
\caption{The sets $C_{12}$\,, $C_{13}$\,, and~$C_{23}$}
\label{Fig:CevaCU}
\end{figure}

\begin{proof}[Proof of Proposition~\ref{P:LatCeva}]
Say that a member of~$\cO(\OK)$ is \emph{initial} if it has the form $\co{0,z}$ for some $z\in\Bbbk^{++}$.

\setcounter{claim}{0}

\begin{claim}\label{Cl:U23init}
The set~$U_{23}$ is initial.
\end{claim}

\begin{cproof}
(cf. Figure~\ref{Fig:U23init}).
{}From Assumption~\eqref{0inU1223} it follows that the leftmost interval of~$U_{12}$ has the form~$\co{0,x}$, where $0<x\leq\infty$.
{}From Assumption~\eqref{U1223notfull} it follows that $x<\infty$, so $x\in\Bbbk^{++}$.
A similar argument applies to~$U_{23}$\,.

Now suppose that~$U_{23}$ is not initial.
{}From Assumptions~\eqref{0inU1223} and~\eqref{U1223notfull} it follows that the second leftmost interval of~$U_{23}$ has one of the the forms $\oo{y,y'}$ or $\oc{y,y'}$ where $0<y<y'\leq\infty$ (and $y'=\infty$ in the second case).
Pick $v\in\oo{y,y'}$\,; observe that $v\in U_{23}$\,.
The element $u\eqdef xyv^{-1}$ belongs to $\oo{0,x}$, thus to~$U_{12}$\,; whence $(1,u,xy)\in C_{12}$\,.
Moreover, the element $u^{-1}xy=v$ belongs to~$U_{23}$\,, thus $(1,u,xy)\in C_{23}$\,.
Using Assumption~\eqref{CevaIneq}, follows that $(1,u,xy)\in C_{13}$\,, that is, $xy\in U_{13}$\,.
It follows that $(1,x,xy)\in C_{13}$\,, thus, by Assumption~\eqref{CevaIneq}, either $(1,x,xy)\in C_{12}$ or $(1,x,xy)\in C_{23}$\,.
In the first case, $x\in U_{12}$\,, a contradiction.
In the second case, $y\in U_{23}$\,, a contradiction.
\end{cproof}
\begin{figure}[htb]
{\centering
\begin{tikzpicture}
\draw (0,0)--(6,0)--(3,6)--(0,0);
\draw (0,0)--(5,2);
\draw (0,0)--(4,4);
\draw (2,0)--(3,6);
\draw (4,0)--(3,6);
\draw (1.5,3)--(6,0);
\draw (12/5,24/5)--(6,0);
\mycircle{(0,0)};
\mycircle{(2,0)};
\mycircle{(4,0)};
\mycircle{(6,0)};
\mycircle{(5,2)};
\mycircle{(4,4)};
\mycircle{(3,6)};
\mycircle{(1.5,3)};
\mycircle{(12/5,24/5)};
\mycircle{(12/5,12/5)};
\mycircle{(15/4,3/2)};
\node[above] at (3,6) {$\seq{0,0,1}$};
\node at (-.8,0) {$\seq{1,0,0}$};
\node at (6.8,0) {$\seq{0,1,0}$};
\node[below] at (2,0) {$\seq{1,u,0}$};
\node[below] at (4,0) {$\seq{1,x,0}$};
\node at (5.8,2) {$\seq{0,1,y}$};
\node at (4.8,4) {$\seq{0,1,v}$};
\node at (.7,3) {$\seq{1,0,xy}$};
\node at (1.6,24/5) {$\seq{1,0,xv}$};
\node[fill=white,below] at (2.05,2.1) {$\seq{1,u,xy}$};
\node[fill=white,below] at (3.9,1.2) {$\seq{1,x,xy}$};
\end{tikzpicture}
}
\caption{Illustrating the proof of Proposition~\ref{P:LatCeva}, Claim~\ref{Cl:U23init}}
\label{Fig:U23init}
\end{figure}
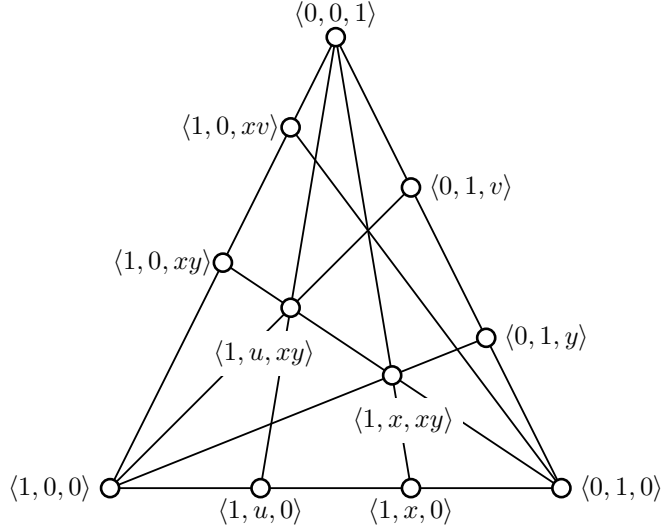

{}From now on we shall write $U_{23}=\co{0,y}$ where $y\in\Bbbk^{++}$.

\begin{claim}\label{Cl:U12init}
The set~$U_{12}$ is initial.
\end{claim}

\begin{cproof}
(cf. Figure~\ref{Fig:U12init}).
Suppose that~$U_{12}$ is not initial.
{}From Assumptions~\eqref{0inU1223} and~\eqref{U1223notfull} it follows that the second leftmost interval of~$U_{12}$ has one of the forms $\oo{x,x'}$ or $\oc{x,x'}$ where $0<x<x'\leq\infty$ (and $x'=\infty$ in the second case).
Pick any $u\in\oo{x,x'}$\,; observe that $u\in U_{12}$\,, that is, $(1,u,xy)\in C_{12}$\,.
The element $v\eqdef u^{-1}xy$ belongs to $\oo{0,y}$, thus to~$U_{23}$\,; that is, $(1,u,xy)\in C_{23}$\,.
Using Assumption~\eqref{CevaIneq}, it follows that $(1,u,xy)\in C_{13}$\,, whence also $(1,x,xy)\in C_{13}$\,.
By Assumption~\eqref{CevaIneq} again, it follows that either $(1,x,xy)\in C_{12}$ or $(1,x,xy)\in C_{23}$\,.
In the first case, $x\in U_{12}$\,, a contradiction.
In the second case, $y\in U_{23}$\,, a contradiction.
\end{cproof}

{}From now on we shall write $U_{12}=\co{0,x}$ where $x\in\Bbbk^{++}$.

\begin{figure}[htb]
{\centering
\begin{tikzpicture}
\draw (0,0)--(6,0)--(3,6)--(0,0);
\draw (0,0)--(4,4);
\draw (0,0)--(5,2);
\draw (2,0)--(3,6);
\draw (4,0)--(3,6);
\draw (1.5,3)--(6,0);
\mycircle{(0,0)};
\mycircle{(2,0)};
\mycircle{(4,0)};
\mycircle{(6,0)};
\mycircle{(3,6)};
\mycircle{(1.5,3)};
\mycircle{(5,2)};
\mycircle{(4,4)};
\mycircle{(15/4,3/2)};
\mycircle{(12/5,12/5)};
\node[above] at (3,6) {$\seq{0,0,1}$};
\node at (-.8,0) {$\seq{1,0,0}$};
\node at (6.8,0) {$\seq{0,1,0}$};
\node[below] at (2,0) {$\seq{1,x,0}$};
\node[below] at (4,0) {$\seq{1,u,0}$};
\node at (5.8,2) {$\seq{0,1,v}$};
\node at (4.8,4) {$\seq{0,1,y}$};
\node[fill=white,below] at (2.05,2.1) {$\seq{1,x,xy}$};
\node[fill=white,below] at (3.9,1.2) {$\seq{1,u,xy}$};
\node at (.7,3) {$\seq{1,0,xy}$};
\end{tikzpicture}
}
\caption{Illustrating the proof of Proposition~\ref{P:LatCeva}, Claim~\ref{Cl:U12init}}
\label{Fig:U12init}
\end{figure}
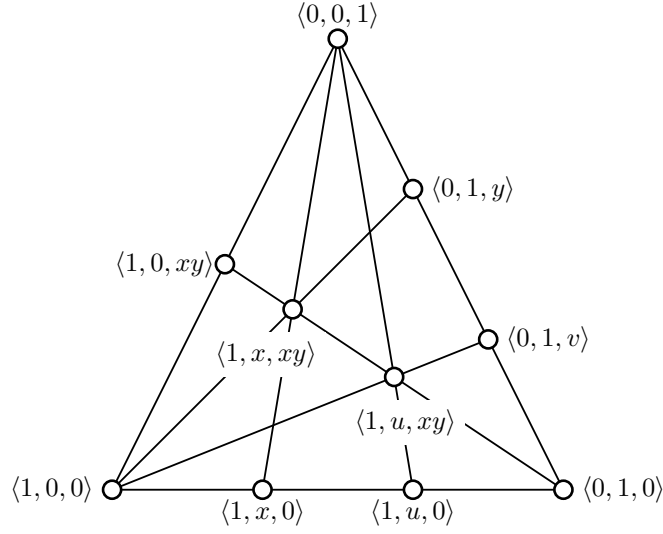

\begin{claim}\label{Cl:0xyciU13}
The set~$U_{13}$ contains $\co{0,xy}$.
\end{claim}

\begin{cproof}
(cf. Figure~\ref{Fig:0xyciU13}).
We need to prove that every element $t\in\co{0,xy}$ belongs to~$U_{13}$\,.
Let first $t=0$.
We need to prove that $(1,0,0)\in C_{13}$\,, which holds owing to Assumption~\eqref{0inU1223}.
Suppose from now on that $t>0$.
There are $u\in\oo{0,x}$ and $v\in\oo{0,y}$ such that $t=uv$.
Observe that $u\in U_{12}$ and $v\in U_{23}$\,.
It follows that $(1,u,uv)\in C_{12}\cap C_{23}$\,, thus, by Assumption~\eqref{CevaIneq}, $(1,u,uv)\in C_{13}$\,, that is, $t\in U_{13}$\,.
\end{cproof}

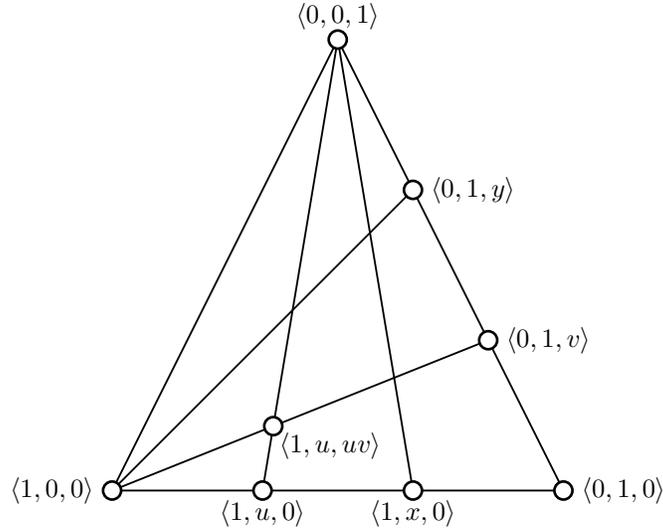
\begin{figure}[htb]
{\centering
\begin{tikzpicture}
\draw (0,0)--(6,0)--(3,6)--(0,0);
\draw (0,0)--(4,4);
\draw (0,0)--(5,2);
\draw (2,0)--(3,6);
\draw (4,0)--(3,6);
\mycircle{(0,0)};
\mycircle{(2,0)};
\mycircle{(4,0)};
\mycircle{(6,0)};
\mycircle{(3,6)};
\mycircle{(5,2)};
\mycircle{(4,4)};
\mycircle{(15/7,6/7)};
\node[above] at (3,6) {$\seq{0,0,1}$};
\node at (-.8,0) {$\seq{1,0,0}$};
\node at (6.8,0) {$\seq{0,1,0}$};
\node[below] at (2,0) {$\seq{1,u,0}$};
\node[below] at (4,0) {$\seq{1,x,0}$};
\node at (5.8,2) {$\seq{0,1,v}$};
\node at (4.8,4) {$\seq{0,1,y}$};
\node[below] at (2.9,.9) {$\seq{1,u,uv}$};
\end{tikzpicture}
}
\caption{Illustrating the proof of Proposition~\ref{P:LatCeva}, Claim~\ref{Cl:0xyciU13}}
\label{Fig:0xyciU13}
\end{figure}

\begin{claim}\label{Cl:U13=0xy}
$U_{13}=\co{0,xy}$.
\end{claim}

\begin{cproof}
(cf. Figure~\ref{Fig:U13=0xy}).
Suppose that there exists $z\in U_{13}\cap\co{xy,\infty}$.
Then $(1,x,z)$ belongs to~$C_{13}$\,, thus, by Assumption~\eqref{CevaIneq}, either to~$C_{12}$ or to~$C_{23}$\,.
In the first case, $x\in U_{12}$\,, a contradiction.
In the second case, $x^{-1}z\in U_{23}$\,, thus $x^{-1}z<y$, a contradiction.
This completes the proof that $U_{13}\subseteq\co{0,xy}$.
Now apply Claim~\ref{Cl:0xyciU13}.
\end{cproof}

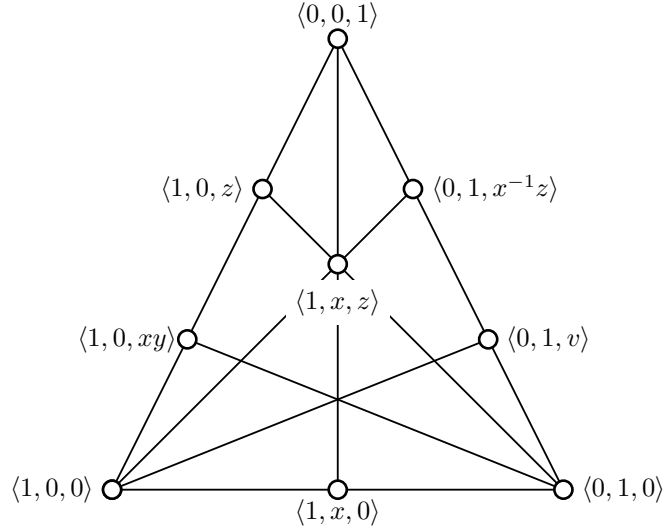
\begin{figure}[htb]
{\centering
\begin{tikzpicture}
\draw (0,0)--(6,0)--(3,6)--(0,0);
\draw (0,0)--(5,2);
\draw (0,0)--(4,4);
\draw (1,2)--(6,0);
\draw (2,4)--(6,0);
\draw (3,0)--(3,6);
\mycircle{(0,0)};
\mycircle{(0,0)};
\mycircle{(3,0)};
\mycircle{(6,0)};
\mycircle{(5,2)};
\mycircle{(4,4)};
\mycircle{(1,2)};
\mycircle{(2,4)};
\mycircle{(3,3)};
\mycircle{(3,6)};
\node[above] at (3,6) {$\seq{0,0,1}$};
\node at (-.8,0) {$\seq{1,0,0}$};
\node at (6.8,0) {$\seq{0,1,0}$};
\node[below] at (3,0) {$\seq{1,x,0}$};
\node at (5.8,2) {$\seq{0,1,v}$};
\node at (5.1,4) {$\seq{0,1,x^{-1}z}$};
\node at (.2,2) {$\seq{1,0,xy}$};
\node at (1.2,4) {$\seq{1,0,z}$};
\node[fill=white,below] at (3,2.8) {$\seq{1,x,z}$};
\end{tikzpicture}
}
\caption{Illustrating the proof of Proposition~\ref{P:LatCeva}, Claim~\ref{Cl:U13=0xy}}
\label{Fig:U13=0xy}
\end{figure}

The combination of Claims~\ref{Cl:U23init}--\ref{Cl:U13=0xy} entails the conclusion of Proposition~\ref{P:LatCeva}.
\end{proof}

\section{The non-commutative diagram~$\vec{A}$}\label{S:vecA}

In this section we shall introduce a non-com\-mu\-ta\-tive diagram (cf. Subsection~\ref{Su:NonCommDiag}), denoted by~$\vec{A}$, of Abelian \lgrp{s} and \lhom{s}, indexed by the cube
 \begin{equation}\label{Eq:Pow3}
 \Pow[3]=\set{\es,1,2,3,12,13,23,123}\,,\quad
 \text{endowed with set inclusion}
 \end{equation}
(where $[3]=\set{1,2,3}$, $12=21=\set{1,2}$, and so on).

We define~$A_{123}$ as the Abelian \lgrp\ defined by the generators~$a$, $a'$, $b$, $c$ subjected to the relations
 \[
 0\leq a\leq a'\leq 2a\,;\quad 0\leq b\,;\quad 0\leq c\,.
 \]
Further, we define the following \lsgrp{s} of~$A_{123}$\,:
\begin{itemize}
\item $A_{12}$ is the \lsgrp\ of~$A_{123}$ generated by $\set{a,b}$;

\item $A_{13}$ is the \lsgrp\ of~$A_{123}$ generated by $\set{a',c}$;

\item $A_{23}$ is the \lsgrp\ of~$A_{123}$ generated by $\set{b,c}$;

\item $A_{1}$ is the \lsgrp\ of~$A_{123}$ generated by $\set{a}$;

\item $A_{2}$ is the \lsgrp\ of~$A_{123}$ generated by $\set{b}$;

\item $A_{3}$ is the \lsgrp\ of~$A_{123}$ generated by $\set{c}$;

\item $A_{\es}=\set{0}$.
\end{itemize}
It is easy to see that each~$A_p$\,, for $p\in\Pow[3]$, can also be defined by generators and relations in a natural way; for example~$A_{12}$ is the Abelian \lgrp\ defined by the generators~$a$, $b$ subjected to the relations $0\leq a$ and $0\leq b$, and so on.
In particular, $A_1\cong A_2\cong A_3\cong\ZZ$.
The diagram~$\vec{A}$ is the $\Pow[3]$-indexed diagram of Abelian \lgrp{s}, whose vertices are the~$A_p$ where $p\in\Pow[3]$ and whose arrows are the loops at every vertex together with the following \lemb{s}:

\begin{itemize}
\item For every $p\in\Pow[3]$, $\vec{A}(\es,p)$ consists of the zero map~$\ga_\es^p$ from~$A_\es$ to~$A_p$\,.

\item For all distinct $i,j\in[3]$, $\vec{A}(i,ij)$ consists of the single map~$\ga_i^{ij}$, defined as the inclusion map from~$A_i$ into~$A_j$\,, \emph{except in case $i=1$ and $j=3$}, in which case~$\ga_i^{ij}=\ga_1^{13}$ is the unique \lhom\ sending~$a$ to~$a'$.
We emphasize this by marking the arrow~$\ga_1^{13}$ with a thick line on Figure~\ref{Fig:DiagA}.

\item For all distinct $i,j\in[3]$, $\vec{A}(ij,123)$ consists of the single map~$\ga_{ij}^{123}$, which is the inclusion map from~$A_{ij}$ into~$A_{123}$\,.

\item $\vec{A}(2,123)$ consists of the single map $\ga_2^{123}=\ga_{12}^{123}\circ\ga_2^{12}=\ga_{23}^{123}\circ\ga_2^{23}$, which is also the inclusion map from~$A_2$ into~$A_{123}$\,.

\item $\vec{A}(3,123)$ consists of the single map $\ga_3^{123}=\ga_{13}^{123}\circ\ga_3^{13}=\ga_{23}^{123}\circ\ga_3^{23}$, which is also the inclusion map from~$A_3$ into~$A_{123}$\,.

\item $\vec{A}(1,123)$ consists of the two \emph{distinct} maps $\ga_{12}^{123}\circ\ga_1^{12}$ (which is also the inclusion map) and~$\ga_{13}^{123}\circ\ga_1^{13}$ (which sends~$a$ to~$a'$) from~$A_1$ into~$A_{123}$\,.

\end{itemize}

The diagram~$\vec{A}$ is partly represented in Figure~\ref{Fig:DiagA}.
On that picture, each node is highlighted by its canonical generators: for example, $A_{123}$ is marked $A_{123}(a,a',b,c)$, $A_{13}$ is marked $A_{13}(a',c)$, and so on.

\begin{figure}[htb]
\begin{tikzcd}
\centering
& A_{123}(a,a',b,c) &&\\
&&&\\
A_{12}(a,b)\ar[uur, "\ga_{12}^{123}" description] &
A_{13}(a',c)\ar[uu, "\ga_{13}^{123}" description] &
A_{23}(b,c)\ar[uul, "\ga_{23}^{123}" description]&
\\
&&&\\
A_1(a)
\arrow[uu,"\ga_1^{12}" description]
\arrow[uur, near start, thick, "\boldsymbol{\ga_1^{13}}" description]
& A_2(b)
\arrow[uul, near start, "\ga_2^{12}" description,crossing over]
& A_3(c)
\arrow[uul, near start, "\ga_3^{13}" description]
\arrow[uu, "\ga_3^{23}" description]
\arrow[from=l,uu, near start, "\ga_2^{23}" description,crossing over]
&
\!\!\!\!\!\!\!\!\!
\text{Each }\ga_\es^{p}\text{ is the zero map }A_\es\to A_p\\
& A_{\es}=\set{0}\arrow[ul]\arrow[u]\arrow[ur] &&
\end{tikzcd}
\caption{The non-com\-mu\-ta\-tive diagram~$\vec{A}$}
\label{Fig:DiagA}
\end{figure}
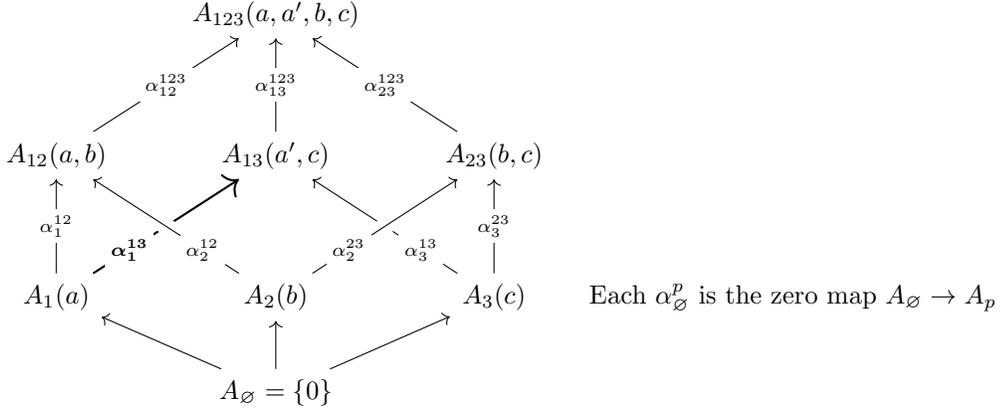

We emphasize that the diagram~$\vec{A}$ of Abelian \lgrp{s} is \emph{not} commutative (for $\vec{A}(1,123)$ has two elements).
However, it has the following remarkable property, involving the construction of the $I$-th power of a diagram (cf. Definition~\ref{D:D^I}), which we will fully bring to use in~\cite{NonElt}.

\begin{proposition}\label{P:vecAIdcComm}
For every set~$I$, the $\Pow[3]$-indexed diagram~$\Idc\vec{A}^I$ is a commutative diagram of completely normal distributive lattices with zero and closed $0$-lattice homomorphisms.
\end{proposition}

\begin{proof}
Although~$\vec{A}$ is not a commutative diagram, it only barely fails to be so since its only non-com\-mu\-ta\-tive square is $(1,12,13,123)$, and then the inequalities $a\leq a'\leq 2a$ yield the statement
 \[
 (\ga_{12}^{123}\circ\ga_1^{12})(x)\leq
 (\ga_{13}^{123}\circ\ga_1^{13})(x)\leq
 2(\ga_{12}^{123}\circ\ga_1^{12})(x)
 \quad\text{whenever }x\in A_1\,.
 \]
It follows that for all $p\leq q$ in~$\Pow[3]$ and all $f,g\colon A_p\to A_q$ in~$\vec{A}$, the statement
 \[
 f(x)\leq 2g(x)\text{ and }g(x)\leq 2f(x)\,,\quad
 \text{for every }x\in A_p\,,
 \]
which we shall denote by $f\asymp_2g$, holds.
It follows easily that if $p_i\leq q_i$ in~$\Pow[3]$ and $f_i\,,g_i\colon A_{p_i}\to A_{q_i}$ whenever $i\in I$, then $\prod_{i\in I}f_i\asymp_2\prod_{i\in I}g_i$\,.
Hence, $\Idc\pI{\prod_{i\in I}f_i}=\Idc\pI{\prod_{i\in I}g_i}$, so there is exactly one arrow from $\prod_{i\in I}A_{p_i}$ to $\prod_{i\in I}A_{q_i}$ in $\vec{A}^I$.

Finally, every $\vec{A}^I(p)$, for $p\in P^I$, is an Abelian \lgrp, thus  $\Idc\vec{A}^I(p)$ is a completely normal distributive lattice with zero.
Whenever $p\leq q$ in~$P^I$, every member of $\vec{A}^I(p,q)$ is an \lhom, thus, by \cite[Proposition~2.6]{MV1}, every member of $\Idc\vec{A}^I(p,q)$ is a closed $0$-lattice homomorphism.
\end{proof}

In the present paper we will only need the case where~$I$ is a singleton:

\begin{corollary}\label{C:vecAIdcComm}
The diagram $\vec{\bA}\eqdef\Idc\vec{A}$ is commutative diagram of completely normal distributive lattices with zero and closed $0$-lattice homomorphisms.
\end{corollary}

For $p\leq q$ in~$\Pow[3]$, we shall denote by~$\bga_p^q$ the unique arrow from~$\bA_p$ to~$\bA_q$ in~$\vec{\bA}$.
For example, $\bga_1^{123}=\Idc(\ga_{12}^{123}\circ\ga_1^{12})=\Idc(\ga_{13}^{123}\circ\ga_1^{13})$.
The elements
 \begin{equation}\label{Eq:ba1ba2ba3}
 \ba_1\eqdef\seq{a}_{A_{123}}=\seq{a'}_{A_{123}}\,,
 \quad\ba_2\eqdef\seq{b}_{A_{123}}\,,
 \quad\ba_3\eqdef\seq{c}_{A_{123}}
 \end{equation}
all belong to $\bA_{123}$\,.

Our main technical lemma is the following.

\begin{lemma}\label{L:IdcANonRepr}
There is no family $\vecm{\bc_{ij}}{i\neq j\text{ in }[3]}$ of elements of $\bA_{123}$ satisfying the following statements:
\begin{enumerater}
\item\label{cijinrng}
Each $\bc_{ij}$ belongs to the range of~$\bga_{ij}^{123}$\,.

\item\label{cijsplits1}
$\ba_i\leq\ba_j\vee\bc_{ij}$ whenever $\set{i,j}$ is either $\set{1,2}$ or $\set{2,3}$.

\item\label{cijsplits2}
$\bc_{ij}\wedge\bc_{ji}=0$ whenever $\set{i,j}$ is either $\set{1,2}$ or $\set{2,3}$.

\item\label{cijCeva}
$\bc_{12}\wedge\bc_{23}\leq
\bc_{13}\leq\bc_{12}\vee\bc_{23}$\,.


\end{enumerater}
\end{lemma}

\begin{proof}
We shall introduce a $\Pow[3]$-indexed commutative diagram~$\vec{\bD}$ of bounded distributive lattices with $0$-lattice homomorphisms.
We set $\bD_\es\eqdef\set{0}$, $\bD_1=\bD_2=\bD_3=\two\eqdef\set{0,1}$, and $\bD_p\eqdef\cO_k$ (cf. Subsection~\ref{Su:OpPolCones}) whenever $p\in\set{12,13,23,123}$ has~$k$ elements.
Each~$\bgd_p^p$ is the identity map on~$\bD_p$\,.
For $p<q$ in~$\Pow[3]$, the map $\bgd_p^q\colon\bD_p\to\bD_q$ is defined as follows:
\begin{itemize}
\item
If $p=\es$ we have no choice, namely $\bgd_\es^p=0$.

\item $\bgd_1^{12}(1)=\bgd_1^{13}(1)=\bgd_2^{23}(1)\eqdef
\bck{\vx_1>0}_2$\,.

\item $\bgd_2^{12}(1)=\bgd_3^{13}(1)=\bgd_3^{23}(1)\eqdef
\bck{\vx_2>0}_2$\,.

\item $\bgd_{ij}^{123}(X)\eqdef
\setm{(x_1,x_2,x_3)\in(\QQ^+)^3}{(x_i\,,x_j)\in X}$, whenever $1\leq i<j\leq 3$ and $X\in\cO_2$\,.

\item $\bgd_i^{123}(1)=\bck{\vx_i>0}_3$ whenever $i\in[3]$.
\end{itemize}

We represent the diagram~$\vec{\bD}$ in Figure~\ref{Fig:DiagD}.

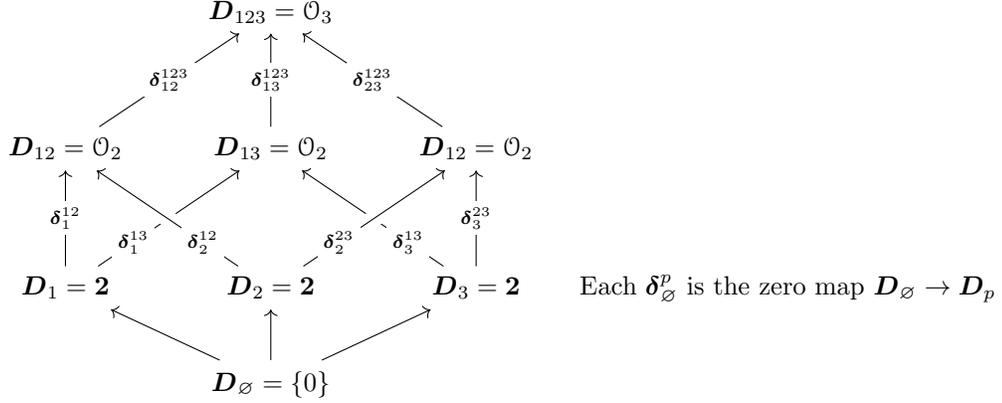
\begin{figure}[htb]
\begin{tikzcd}
& \bD_{123}=\cO_3 &&\\
&&&\\
\bD_{12}=\cO_2\arrow[uur, "\bgd_{12}^{123}" description] &
\bD_{13}=\cO_2\arrow[uu, "\bgd_{13}^{123}" description] &
\bD_{12}=\cO_2\arrow[uul, "\bgd_{23}^{123}" description]&
\\
&&&\\
\bD_1=\two
\arrow[uu,"\bgd_1^{12}" description]
\arrow[uur, near start, "\bgd_1^{13}" description]
& \bD_2=\two
\arrow[uul, near start, "\bgd_2^{12}" description,crossing over]
& \bD_3=\two
\arrow[uul, near start, "\bgd_3^{13}" description]
\arrow[uu, "\bgd_3^{23}" description]
\arrow[from=l,uu, near start, "\bgd_2^{23}" description,crossing over]
&
\!\!\!\!\!\!\!\!\!
\text{Each }\bgd_\es^{p}\text{ is the zero map }\bD_\es\to \bD_p\\
& \bD_{\es}=\set{0}\arrow[ul]\arrow[u]\arrow[ur] &&
\end{tikzcd}
\caption{The diagram~$\vec{\bD}$}\label{Fig:DiagD}
\end{figure}

The verification of the following claim is straightforward.

\setcounter{claim}{0}

\begin{claim}\label{Cl:vecDcomm}
$\vec{\bD}$ is a commutative diagram of bounded distributive lattices with $0$-lattice homomorphisms.
\end{claim}

Define $0$-lattice homomorphisms $\eta_p\colon\bA_p\to \bD_p$\,, for $p\in\Pow[3]$, as follows:
\begin{itemize}
\item Whenever $p\in\set{\es,1,2,3}$, $\eta_p$ is the unique isomorphism from~$\bA_p$ onto~$\bD_p$\,.

\item
Now we describe $\eta_p$ for $p\in\set{12,13,23}$\,:
\[
\eta_{12}\pI{\seq{t(a,b)}_{A_{12}}}=
\eta_{13}\pI{\seq{t(a',c)}_{A_{13}}}=
\eta_{23}\pI{\seq{t(b,c)}_{A_{23}}}
\eqdef\bck{t(\vx_1,\vx_2)\neq0}_2\,,
\]
for every binary $\ell$-term~$t$.

\item
We finally describe $\eta_{123}$\,:
 \[
 \eta_{123}\pI{\seq{t(a,a',b,c)}_{A_{123}}}\eqdef
 \bck{t(\vx_1,\vx_1,\vx_2,\vx_3)\neq0}_3\,,
 \]
for every $4$-ary $\ell$-term~$t$.
This makes sense because for all $x_1,x_2,x_3\in\QQ^+$, the quadruple $(x_1,x_1,x_2,x_3)$ satisfies the defining relations of~$A_{123}$\,.
\end{itemize}

\begin{claim}\label{Cl:etaNatTrans}
The family $\vec{\eta}\eqdef\vecm{\eta_p}{p\in\Pow[3]}$ is a natural transformation from~$\vec{\bA}$ to~$\vec{\bD}$.
Furthermore, $\eta_p$ is an isomorphism whenever $p\neq123$.
\end{claim}

\begin{cproof}
The statement about isomorphisms easily follows from the Baker-Beynon duality for finitely presented Abelian \lgrp{s} (cf. Baker~\cite{Baker1968}, Beynon~\cite{Beyn1975}).

Now in order to verify that~$\vec{\eta}$ is a natural transformation, it suffices to prove that $\eta_q\circ\bga_p^q=\bgd_p^q\circ\eta_p$ whenever~$p$ is a lower cover of~$q$ in~$\Pow[3]$.
This is trivial if $p=\es$, in which case both composed maps are zero.
If $p=1$ and $q=13$, we compute
 \begin{align*}
 (\eta_{13}\circ\bga_1^{13})\pI{\seq{a}_{A_{1}}}
 &=\eta_{13}\pI{\seq{a'}_{A_{13}}}
 =\bck{\vx_1>0}_2\,,\\
 (\bgd_1^{13}\circ\eta_1)\pI{\seq{a}_{A_{1}}}&=\bgd_1^{13}(1)
 =\bck{\vx_1>0}_2\,,
 \end{align*}
so we are done in that case.
The other cases, where~$p$ has one element and~$q$ two elements, are handled similarly.

If $p=13$ and $q=123$, we compute, for every binary $\ell$-term~$t$,
 \begin{align*}
 (\eta_{123}\circ\bga_{13}^{123})\pI{\seq{t(a',c)}_{A_{13}}}&=
 \eta_{123}\pI{\seq{t(a',c)}_{A_{123}}}=
 \bck{t(\vx_1,\vx_3)\neq0}_3\,,
 \\
 (\bgd_{13}^{123}\circ\eta_{13})\pI{\seq{t(a',c)}_{A_{13}}}&=
 \bgd_{13}^{123}\pI{\bck{t(\vx_1,\vx_2)\neq0}_2}
 =\bck{t(\vx_1,\vx_3)\neq0}_3\,,
  \end{align*}
so we are done in that case.
The other cases, where $p\in\set{12,23}$ and $q=123$, are handled similarly.
\end{cproof}

Now we argue by contradiction, by supposing that the~$\bc_{ij}$ satisfy Conditions \eqref{cijinrng}--\eqref{cijCeva} from the statement of Lemma~\ref{L:IdcANonRepr}.
For $i\neq j$ in~$[3]$, it follows from Condition~\eqref{cijinrng} that $\bc_{ij}=\seq{c_{ij}}_{A_{123}}$ for some $c_{ij}\in A_{ij}^+$\,.
The set $C'_{ij}\eqdef\eta_{ij}\pI{\seq{c_{ij}}_{A_{ij}}}$ belongs to $\bD_{ij}=\cO_2$\,, thus it is determined by its intersection with the segment ($1$-simplex) $\setm{(x,y)\in(\QQ^+)^2}{x+y=1}$, which is a finite union of relatively open intervals of that segment.
Hence, there exists $U_{ij}\in\cO(\OQ)$ such that
 \begin{equation}\label{Eq:Uij2C'ij}
 C'_{ij}=\setm{(x,y)\in(\QQ^+)^2\setminus\set{(0,0)}}
 {x^{-1}y\in U_{ij}}\,.
 \end{equation}
Setting $C_{ij}\eqdef\eta_{123}(\bc_{ij})=\bgd_{ij}^{123}(C'_{ij})$, we get
 \begin{equation}\label{Eq:Uij2Cij}
 C_{ij}=\begin{cases}
 \setm{(x_1,x_2,x_3)\in(\QQ^+)^3}
 {(x_i,x_j)\neq(0,0)\text{ and }
 x_i^{-1}x_j\in U_{ij}}&\text{if }i<j\,,\\
 \setm{(x_1,x_2,x_3)\in(\QQ^+)^3}
 {(x_i,x_j)\neq(0,0)\text{ and }
 x_j^{-1}x_i\in U_{ij}}&\text{if }i>j\,.
 \end{cases}
 \end{equation}
By applying the homomorphism~$\eta_{123}$ to Conditions \eqref{cijsplits1}--\eqref{cijCeva}, we thus obtain, setting $P_i\eqdef\bck{\vx_i>0}_3$\,, the following relations:
\begin{itemize}
\item[(C1)]
$P_i\subseteq P_j\cup C_{ij}$ whenever $\set{i,j}$ is either $\set{1,2}$ or $\set{2,3}$.

\item[(C2)]
$C_{ij}\cap C_{ji}=\es$ whenever $\set{i,j}$ is either $\set{1,2}$ or $\set{2,3}$.

\item[(C3)]
$C_{12}\cap C_{23}\subseteq
C_{13}\subseteq C_{12}\cup C_{23}$\,.

\end{itemize}
By~(C1) and since $(1,0,0)\in P_1\setminus P_2$\,, we get $(1,0,0)\in C_{12}$\,, that is (cf.~\eqref{Eq:Uij2Cij}), $0\in U_{12}$\,.
Similar arguments yield the relations $0\in U_{23}$ and $0\in U_{13}$\,.
Similarly, since $(0,1,0)\in P_2\setminus P_1$ and by~(C1), we get  $(0,1,0)\in C_{21}$\,, that is (cf.~\eqref{Eq:Uij2Cij}), $\infty\in U_{21}$\,.
Since $U_{21}\in\cO(\OQ)$, it follows that~$U_{21}$ contains an interval of the form $[z,\infty]$.
{}From~(C2) it follows that $U_{12}\cap U_{21}=\es$, thus~$U_{12}$ is a bounded subset of~$\QQ^+$.
We thus have proved that \emph{$U_{12}$ is a bounded subset of~$\QQ^+$ containing~$0$ as an element}.
By a similar argument, \emph{$U_{23}$ is a bounded subset of~$\QQ^+$ containing~$0$ as an element}.
Therefore, the assumptions of Proposition~\ref{P:LatCeva} are satisfied, so there are $\gl,\mu\in\QQ^{++}$ such that
 \begin{equation}\label{Eq:UijAllInit}
 U_{12}=\co{0,\gl}\,,\quad U_{23}=\co{0,\mu}\,,\text{ and }
 U_{13}=\co{0,\gl\mu}\,.
 \end{equation}
{}From $U_{12}=\co{0,\gl}$ and~\eqref{Eq:Uij2C'ij} it follows that
 \begin{align*}
 \eta_{12}\pI{\seq{c_{12}}_{A_{12}}}&=
 \setm{(x_1,x_2)\in(\QQ^+)^2\setminus\set{(0,0)}}
 {x_1^{-1}x_2<\gl}\\
 &=\bck{\gl\vx_1-\vx_2>0}\\
 &=\eta_{12}\pI{\seq{(\gl a-b)^+}_{A_{12}}}\,. 
 \end{align*}
Since~$\eta_{12}$ is an isomorphism (cf. Claim~\ref{Cl:etaNatTrans}), it follows that $\seq{c_{12}}_{A_{12}}=\seq{(\gl a-b)^+}_{A_{12}}$, that is,
 \begin{equation}\label{Eq:mu2c12}
 c_{12}\asymp(\gl a-b)^+\text{ within }A_{12}\,.
 \end{equation}
Similar arguments, using~\eqref{Eq:UijAllInit}, yield the relations
 \begin{equation}\label{Eq:glmu2c2313}
 c_{23}\asymp(\mu b-c)^+\text{ within }A_{23}\quad\text{and}\quad
 c_{13}\asymp(\gl\mu a'-c)^+\text{ within }A_{13}\,.
 \end{equation}
Condition~\eqref{cijCeva}, together with~\eqref{Eq:mu2c12} and~\eqref{Eq:glmu2c2313}, thus yields
 \begin{equation}\label{Eq:cijpropto}
 (\gl a-b)^+\wedge(\mu b-c)^+\propto(\gl\mu a'-c)^+\propto
 (\gl a-b)^+\vee(\mu b-c)^+\text{ within }A_{123}\,.
 \end{equation}
Since the quadruple $(1,2,\gl,\gl\mu)$ of rational numbers satisfies the defining relations of~$A_{123}$\,, there exists a unique \lhom\ $f\colon A_{123}\to\QQ$ sending $(a,a',b,c)$ to $(1,2,\gl,\gl\mu)$.
By applying~$f$ to the right hand side inequality of~\eqref{Eq:cijpropto}, we obtain that $\gl\mu=(2\gl\mu-\gl\mu)^+\propto0$, a contradiction.
\end{proof}

\section{Cevian operations}
\label{S:CevaDiff}

In this section we shall define \emph{Cevian operations} on certain distributive lattices with zero.
The existence of a Cevian operation is a strong form of complete normality.
It will turn out that such operations exist on all lattices of the form~$\Csc{G}$ (cf. Proposition~\ref{P:CevaSclgrp}) or~$\Idc{G}$ where the \lgrp~$G$ is representable (cf. Proposition~\ref{P:CevaIdRepr}).

\begin{definition}\label{D:CevaDiff}
Let~$D$ be a distributive lattice with zero.
A binary operation~$\sd$ on~$D$ is \emph{Cevian} if the following conditions hold:
\begin{itemize}
\item[(Cev1)]
$x\leq y\vee(x\sd y)$ for all $x,y\in D$;

\item[(Cev2)]
$(x\sd y)\wedge(y\sd x)=0$ for all $x,y\in D$;

\item[(Cev3)]
$x\sd z\leq(x\sd y)\vee(y\sd z)$ for all $x,y,z\in D$.
\end{itemize}
A distributive lattice with zero is \emph{Cevian} if it has a Cevian operation.
\end{definition}

Obviously, every Cevian lattice is completely normal.
The main example of Wehrung \cite[\S~6]{MVRS} shows that the converse fails at cardinality~$\aleph_2$\,.
We will shortly see that there is no such
counterexample in the countable case (cf. Corollary~\ref{C:CevaSclgrp}).
We will also find a new completely normal non-Cevian example of cardinality~$\aleph_2$\,, with additional features, in Theorem~\ref{T:Main}.

\begin{lemma}\label{L:Cev4}
Let~$\sd$ be a binary operation on a distributive lattice~$D$ with zero, satisfying both~\textup{(Cev2)} and~\textup{(Cev3)}.
Then $(x\sd y)\wedge(y\sd z)\leq x\sd z$ for all $x,y,z\in D$.
\end{lemma}

\begin{proof}
It follows from (Cev3) that
 \begin{equation}\label{Eq:xyleqxzy}
 x\sd y\leq(x\sd z)\vee(z\sd y)\,.
 \end{equation}
Further, it follows from (Cev2) that $(y\sd z)\wedge(z\sd y)=0$.
Therefore, meeting~\eqref{Eq:xyleqxzy} with~$y\sd z$ and using the distributivity of~$D$, we obtain
 \begin{equation*}
 (x\sd y)\wedge(y\sd z)\leq(x\sd z)\wedge(y\sd z)\leq x\sd z\,.
 \tag*{\qed}
 \end{equation*}
\renewcommand{\qed}{}
\end{proof}

\begin{lemma}\label{L:ProdHom}\hfill
\begin{enumerater}
\item
Any product of a family of Cevian lattices is Cevian.

\item
Any homomorphic image of a Cevian lattice is Cevian.

\item
Any ideal of a Cevian lattice is Cevian.

\end{enumerater}
\end{lemma}

\begin{proof}
\emph{Ad}~(1).
Let~$\sd_i$ be a Cevian operation on~$D_i$ for each $i\in I$.
On the product $D\eqdef\prod_{i\in I}D_i$\,, set $x\sd y\eqdef\vecm{x_i\sd_iy_i}{i\in I}$.

\emph{Ad}~(2).
Let $f\colon D\twoheadrightarrow E$ be a surjective lattice homomorphism and let~$\sd_D$ be a Cevian operation on~$D$.
Then~$E$ is also a distributive lattice with zero.
For each $x\in E$, pick a preimage~$\ol{x}$ of~$x$ under~$f$, and set $x\sd_Ey\eqdef f\pI{\ol{x}\sd_D\ol{y}}$ for all $x,y\in E$.
Then~$\sd_E$ is a Cevian operation on~$E$.

\emph{Ad}~(3).
Say that a Cevian operation~$\sd$ on~$D$ is \emph{normalized} if $x\sd y\leq x$ for all $x,y\in D$.
For every Cevian operation~$\sd$, the variant operation~$\sd'$ defined by
 \[
 x\sd'y\eqdef x\wedge(x\sd y)\,,\quad\text{for all }x,y\in D\,,
 \]
is a normalized Cevian operation on~$D$.
In particular, every ideal~$I$ of~$D$ is closed under~$\sd'$, thus~$\sd'$ defines, by restriction, a (normalized) Cevian operation on~$I$.
\end{proof}

For any elements~$x$ and~$y$ in an \lgrp~$G$, set $x\sd y\eqdef(x-y)^+=x-x\wedge y$; write $x\sd_Gy$ instead of $x\sd y$ if~$G$ needs to be specified.

\begin{lemma}\label{L:Cevalgrp}
The operation~$\sd_G$\,, defined on~$G$, satisfies the statements~\textup{(Cev2)} and~\textup{(Cev3)}, for any \pup{not necessarily Abelian} \lgrp~$G$.
\end{lemma}

\begin{proof}
(Cev2) is easy: $(x\sd y)\wedge(y\sd x)=(x-x\wedge y)\wedge(y-x\wedge y)=x\wedge y-x\wedge y=0$.
For the right hand side inequality of~(Cev3), observe that $x\leq(x\sd y)+y$ and $y\leq(y\sd z)+z$, thus $x\leq(x\sd y)+(y\sd z)+z$, and thus $x-z\leq(x\sd y)+(y\sd z)$.
Since $0\leq(x\sd y)+(y\sd z)$, it follows that $x\sd z=(x-z)^+\leq(x\sd y)+(y\sd z)$.
\end{proof}

\begin{proposition}\label{P:CevaSclgrp}
Let~$G$ be an \lgrp.
Then~$\Csc{G}$ is a Cevian lattice.
\end{proposition}

\begin{proof}
For any $\bx\in\Csc{G}$, pick $\gc(\bx)\in G^+$ such that $\bx=\seq{\gc(\bx)}$ and set
 \[
 \bx\sd\by\eqdef\seq{\gc(\bx)\sd_G\gc(\by)}\,,
 \quad\text{for all }\bx,\by\in\Csc{G}\,.
 \]
Let $\bx,\by,\bz\in\Csc{G}$ with respective images~$x$, $y$, $z$ under~$\gc$.
It follows from the equation $x=(x\sd_Gy)+x\wedge y$ that $\seq{x}\subseteq\seq{y}\vee\seq{x\sd_Gy}$, that is, $\bx\subseteq\by\vee(\bx\sd\by)$.

Using Lemma~\ref{L:Cevalgrp} together with Bigard \emph{et al.} \cite[Proposition~2.2.11]{BKW}, we get
 \[
 (\bx\sd\by)\wedge(\by\sd\bx)=\seq{x\sd_Gy}\cap\seq{y\sd_Gx}
 =\seq{(x\sd_Gy)\wedge(y\sd_Gx)}=0\,.
 \]
By using Lemma~\ref{L:Cevalgrp}, we also get
 \begin{equation*}
 \bx\sd\bz=\seq{x\sd_Gz}\subseteq
 \seq{x\sd_Gy}\vee\seq{y\sd_Gz}
 =(\bx\sd\by)\vee(\by\sd\bz)\,.\tag*{\qed}
 \end{equation*}
\renewcommand{\qed}{}
\end{proof}

The main result of the author's paper~\cite{MV1} states that every countable completely normal distributive lattice with zero is isomorphic to~$\Idc{G}$ for some Abelian \lgrp~$G$.
Consequently, by Proposition~\ref{P:CevaSclgrp}, we get:

\begin{corollary}\label{C:CevaSclgrp}
A countable distributive lattice with zero is Cevian if{f} it is completely normal.
\end{corollary}

\begin{remark}\label{Rk:CevaSclgrp}
The result of Proposition~\ref{P:CevaSclgrp} cannot be extended to~$\Idc{G}$ for arbitrary \lgrp{s}~$G$.
Indeed, we proved in R\r{u}\v{z}i\v{c}ka \emph{et al.} \cite[Theorem~6.3]{RTW} that
\emph{every countable distributive \jzs\ is isomorphic to~$\Idc{G}$ for some \lgrp~$G$}.
In particular, $\Idc{G}$ may fail to be a lattice, and even if it is a lattice, it may fail to be completely normal (consider a square with a new zero element added).
We will see shortly (cf. Proposition~\ref{P:CevaIdRepr}) that this kind of counterexample does not occur within the class of representable \lgrp{s}.

Incidentally, it follows from \cite[Corollary~3.9]{RTW} that \emph{not every distributive \jzs\ is isomorphic to~$\Idc{G}$ for an \lgrp~$G$.}
\end{remark}

Recall that an \lgrp\ is \emph{representable} if it is a subdirect product of totally ordered groups.
Equivalently (cf. Bigard \emph{et al} \cite[Proposition~4.2.9]{BKW}), $G$ satisfies the identity $(2\vx)\wedge(2\vy)=2(\vx\wedge\vy)$.

\begin{lemma}\label{L:Represxmeety}
Let~$x$, $y$, $u$, $v$ be elements in a representable \lgrp~$G$.
Then
 \begin{equation}\label{Eq:ux-uvy-v}
 (u+x-u)\wedge(v+y-v)\leq
 (u+x\wedge y-u)\vee(v+x\wedge y-v)\,.
 \end{equation}
\end{lemma}

\begin{proof}
It suffices to consider the case where~$G$ is totally ordered; so, by symmetry, we may assume that $x\leq y$.
Then the right hand side of~\eqref{Eq:ux-uvy-v} is equal to\linebreak $(u+x-u)\vee(v+x-v)$, which lies above $u+x-u$, thus above $(u+x-u)\wedge(v+y-v)$.
\end{proof}

\begin{lemma}\label{L:lIdRepres}
Let~$G$ be a representable \lgrp\ and let $x,y\in G^+$.
Then $\seql{x}\cap\seql{y}=\seql{x\wedge y}$.
Consequently, $\Idc{G}$ is a distributive lattice.
\end{lemma}

\begin{proof}
We prove the nontrivial containment.
Any element of~$\seql{x}$ lies, in absolute value, below a finite sum of conjugates of~$x$; and similarly for~$\seql{y}$ and~$y$.
By Bigard \emph{et al.} \cite[Th\'eor\`eme~1.2.16]{BKW}, it thus suffices to prove that $(u+x-u)\wedge(v+y-v)$ belongs to $\seql{x\wedge y}$ whenever $u,v\in G$.
This follows immediately from Lemma~\ref{L:Represxmeety}.
\end{proof}

\begin{proposition}\label{P:CevaIdRepr}
Let~$G$ be a representable \lgrp.
Then~$\Idc{G}$ is naturally \pup{in the functorial sense} a homomorphic image of~$\Csc{G}$.
In particular, it is a Cevian lattice.
\end{proposition}

\begin{proof}
By Lemma~\ref{L:ProdHom} and Proposition~\ref{P:CevaSclgrp}, it suffices to prove that~$\Idc{G}$ is a homomorphic image of $\Csc{G}$.
By Lemma~\ref{L:lIdRepres}, the assignment $\seq{x}\mapsto\seql{x}$ defines a \mh\ from~$\Csc{G}$ onto~$\Idc{G}$, and this naturally in~$G$.
It is obviously a surjective \jh.
\end{proof}

\begin{remark}\label{Rk:CevaIdRepr}
We observed in \cite[Example~10.6]{MV1} that the class of all lattices of the form~$\Idc{G}$, with~$G$ an Abelian \lgrp, is not closed under homomorphic images.
Since every Abelian \lgrp\ is representable and by Proposition~\ref{P:CevaIdRepr}, it follows that \emph{not every Cevian distributive lattice with zero is isomorphic to~$\Idc{G}$ for some Abelian \lgrp~$G$}.
\end{remark}

The following result shows that the non-commutativity of the diagram~$\vec{A}$ can be read on the commutative diagram~$\Idc\vec{A}$.

\begin{theorem}\label{T:NotCongIdcvecG}
Let $\vec{G}=\vecm{G_p,\gc_p^q}{p\leq q\text{ in }\Pow[3]}$ be a $\Pow[3]$-indexed commutative diagram of \lgrp{s} and \lhom{s} and let $\vec{\eta}=\vecm{\eta_p}{p\in\Pow[3]}$ be a natural transformation from~$\Idc\vec{G}$ to~$\Idc\vec{A}$.
Then $\eta_i=0$ for some $i\in\set{1,2,3}$.
\end{theorem}

\begin{proof}
Suppose otherwise.
For each $i\in[3]$, there exists $c_i\in G_i^+$ such that $\eta_i(\seq{c_i}_{G_i})\neq0$.
Since $\Idc A_i\cong\set{0,1}$, it follows that
$\eta_i(\seq{c_i}_{G_i})=1$.
Set $b_i\eqdef\gc_i^{123}(\seq{c_i}_{G_i})$ and
 \[
 \bc_{ij}\eqdef\eta_{123}(\seq{b_i\sd_{G_{123}}b_j}_{G_{123}})
 \]
(an element of~$\Idc{A_{123}}$), for all distinct $i,j\in[3]$.
Hence, the element
 \begin{align*}
 \bc_{ij}&=(\eta_{123}\circ\Idc{\gc_{ij}^{123}})
 \pI{\seq{\gc_i^{ij}(g_i)\sd_{G_{ij}}\gc_j^{ij}(g_j)}_{G_{ij}}}\\
 &=(\bga_{ij}^{123}\circ\eta_{ij})
 \pI{\seq{\gc_i^{ij}(g_i)\sd_{G_{ij}}\gc_j^{ij}(g_j)}_{G_{ij}}}
 \end{align*}
belongs to the range of~$\bga_{ij}^{123}$.
Furthermore, for each $i\in[3]$,
 \begin{align*}
 \eta_{123}(\seq{b_i}_{G_{123}})&=
 (\eta_{123}\circ\Idc{\gc_i^{123}})(\seq{b_i}_{G_{i}})\\
 &=(\bga_i^{123}\circ\eta_i)(\seq{b_i}_{G_{i}})\\
 &=\bga_i^{123}(1)\\
 &=\ba_i
 \end{align*}
(we defined the~$\ba_i$ in~\eqref{Eq:ba1ba2ba3}).
By using (the argument of) Proposition~\ref{P:CevaSclgrp}, together with Lemma~\ref{L:Cev4}, it follows that the elements~$\bc_{ij}$\,, where $i\neq j$ in~$[3]$, satisfy Assumptions \eqref{cijsplits1}--\eqref{cijCeva} of the statement of Lemma~\ref{L:IdcANonRepr}; a contradiction.
\end{proof}

By using Proposition~\ref{P:CevaIdRepr}, we thus obtain

\begin{corollary}\label{C:NotCongIdcvecG}
There is no~$\Pow[3]$-indexed commutative diagram~$\vec{G}$ of \lgrp{s} \pup{resp., representable \lgrp{s}} and \lhom{s} such that $\Csc\vec{G}\cong\Idc\vec{A}$ \pup{resp., $\Idc\vec{G}\cong\Idc\vec{A}$}.
\end{corollary}

\section[Condensates]{A crash course on $P$-scaled Boolean algebras and condensates}
\label{S:CrashCondens}

In order to turn the diagram counterexample (Lemma~\ref{L:IdcANonRepr}) to an object counterexample (Theorem~\ref{T:Main}), we will need to apply a complex, technical result of category theory called the \emph{Armature Lemma}, introduced in Gillibert and Wehrung \cite[Lemma~3.2.2]{Larder}.
In order to help the reader get a feel of the machinery underlying that tool, we shall devote this section to giving a flavor of that machinery.

\subsection{$P$-scaled Boolean algebras, normal morphisms, $\two[p]$}
\label{Su:BoolP}
For an arbitrary poset~$P$, a \emph{$P$-scaled Boolean algebra} is a structure
 \[
 \bA=\pI{A,\vecm{A^{(p)}}{p\in P}}\,,
 \]
where~$A$ is a Boolean algebra, every~$A^{(p)}$ is an ideal of~$A$, $A=\bigvee\vecm{A^{(p)}}{p\in P}$ within the ideal lattice of~$A$, and $A^{(p)}\cap A^{(q)}=\bigvee\vecm{A^{(r)}}{r\geq p,q}$ whenever $p,q\in P$.
For $P$-scaled Boolean algebras~$\bA$ and~$\bB$, a \emph{morphism} from~$\bA$ to~$\bB$ is a homomorphism $f\colon A\to B$ of Boolean algebras such that $f[A^{(p)}]\subseteq B^{(p)}$ for every $p\in P$.
If~$f$ is surjective and $f[A^{(p)}]=B^{(p)}$ for every~$p$, we say that~$f$ is \emph{normal}.
The category of all $P$-scaled Boolean algebras is denoted by~$\Bool_P$\,.
It has all small directed colimits and all finite products.

We prove in \cite[Corollary~4.2.7]{Larder} that a $P$-scaled Boolean algebra~$\bA$ is finitely presented (in the sense of Gabriel and Ulmer~\cite{GabUlm},  Ad{\'a}mek and Rosick{\'y}~\cite{AdRo94}) within~$\Bool_P$ if{f}~$A$ is finite and for every atom~$a$ of~$A$, the ideal $\|a\|_{\bA}\eqdef\setm{p\in P}{a\in A^{(p)}}$ has a largest element, then denoted by~$|a|_{\bA}$\,.
Finitely presented members of~$\Bool_P$ approximate well the whole class:
\begin{itemize}
\item[---] (cf. \cite[Proposition~2.4.6]{Larder}) Every member of~$\Bool_P$ is a monomorphic directed colimit of a direct system of finitely presented members of~$\Bool_P$\,.

\item[---] (cf. \cite[Proposition~2.5.5]{Larder}) Every normal morphism in~$\Bool_P$ is a directed colimit, within the category $\Bool_P^\two$ of all arrows of~$\Bool_P$\,, of a direct system of normal morphisms between finitely presented members of~$\Bool_P$\,.
\end{itemize}
For every $p\in P$, we introduced in \cite[Definition~2.6.1]{Larder} the $P$-scaled Boolean algebra
 \[
 \two[p]\eqdef\pI{\two,\vecm{\two[p]^{(q)}}{q\in P}}
 \]
where we define~$\two[p]^{(q)}$ as $\set{0,1}$ if~$q\leq p$, $\set{0}$ otherwise.

\subsection{Norm-coverings, $\go$-lifters, $\xF(X)$, $\pi_x^X$}
\label{Su:NormCov}

Following \cite[Definition~2.1.2]{Larder}, we say that a poset~$P$ is
\begin{itemize}
\item[---] a \emph{\pjs} if the set~$U$ of all upper bounds of any finite subset~$X$ of~$P$ is a finitely generated upper subset of~$P$; then we denote by $\Sor X$ the (finite) set of all minimal elements of~$U$;

\item[---] \emph{supported} if it is a \pjs\ and every finite subset of~$P$ is contained in a finite subset~$Y$ of~$P$ which is \emph{$\Sor$-closed}, that is, $\Sor{Z}\subseteq Y$ whenever~$Z$ is a finite subset of~$Y$;

\item[---] an \emph{\ajs} if it is a \pjs\ in which every principal ideal~$P\dnw a$ is a \js.

\end{itemize}
We observed in \cite[\S~2.1]{Larder} the non-reversible implications
 \[
 \text{\js}\Rightarrow\text{\ajs}\Rightarrow\text{supported}
 \Rightarrow\text{\pjs}\,.
 \]
Following \cite[\S~2.6]{Larder}, a \emph{norm-covering} of a poset~$P$ is a pair $(X,\partial)$ where~$X$ is a \pjs\ and $\partial\colon X\to P$ is an isotone map.
For such a norm-covering, we denote by~$\rF(X)$ the Boolean algebra defined by generators~$\tilde{u}$, where $u\in X$, and relations
 \begin{align*}
 \tilde{v}&\leq\tilde{u}\,,&&\text{whenever }u\leq v\text{ in }X\,;\\
 \tilde{u}\wedge\tilde{v}&=
 \bigvee\vecm{\tilde{w}}{x\in\Sor\set{u,v}}\,,
 &&\text{whenever }u,v\in X\,;\\
 1&=\bigvee\vecm{\tilde{w}}{w\in\Sor\es}\,. 
 \end{align*}
Furthermore, for every $p\in P$, we denote by~$\rF(X)^{(p)}$ the ideal of~$\rF(X)$ generated by $\setm{\tilde{u}}{u\in X\,,\ p\leq\partial u}$.
The structure $\xF(X)\eqdef\pI{\rF(X),\vecm{\rF(X)^{(p)}}{p\in P}}$ is a $P$-scaled Boolean algebra \cite[Lemma~2.6.5]{Larder}.

In~\cite{Larder} we also introduce, for every $x\in X$, the unique morphism $\pi_x^X\colon\xF(X)\to\two[\partial x]$ that sends every $\tilde{u}$, where $u\in X$, to~$1$ if $u\leq x$ and~$0$ otherwise.
This morphism is normal  (cf. \cite[Lemma~2.6.7]{Larder}).

We will also need here a specialization of the concept of \emph{$\gl$-lifter} (obtained by setting $\gl=\aleph_0$ and~$\bX=$ set of all principal ideals of~$X$) introduced in \cite[\S~3.2]{Larder}.

\begin{definition}\label{D:goLift}
A \emph{principal $\go$-lifter} of a poset~$P$ is a norm-covering $(X,\partial)$ of~$P$ such that
\begin{enumerater}
\item
the set $X^=\eqdef
\setm{x\in X}{\partial x\text{ is not maximal in }P}$ is lower finite;

\item
$X$ is supported;

\item
For every map $S\colon X^=\to[X]^{<\go}$, there exists an isotone section~$\gs$ of~$\partial$ such that $S(\gs(p))\cap\gs(q)\subseteq\gs(p)$ for all $p<q$ in~$P$.

\end{enumerater}
\end{definition}

\subsection{The construction $\bA\otimes\vec{S}$}
\label{Su:AotimesS}

Let a category~$\cS$ have all nonempty finite products and all small directed colimits.
Let $\vec{S}=\vecm{S_p,\gs_p^q}{p\leq q\text{ in }P}$ be a $P$-indexed direct system in~$\cS$.
The functor ${}_{-}\otimes\vec{S}\colon\Bool_P\to\cS$ is first defined on all finitely presented members of~$\Bool_P$\,, as follows.
If~$\bA$ is finitely presented, we set
 \[
 \bA\otimes\vec{S}\eqdef\prod\vecm{S_{|a|_{\bA}}}{a\in\At A}\,.
 \]
In particular, $\two[p]\otimes\vec{S}=S_p$\,.
For a morphism $\gf\colon\bA\to\bB$ between finitely presented $P$-scaled Boolean algebras and an atom~$b$ of~$B$, we define~$b^\gf$ as the unique atom of~$A$ such that $b\leq\gf(b^\gf)$.
Then the product morphism $\gf\otimes\vec{S}\eqdef\prod\Vecm{\gs_{|b^\gf|_{\bA}}^{|b|_{\bB}}}{b\in\At B}$ goes from~$\bA\otimes\vec{S}$ to $\bB\otimes\vec{S}$.
This defines a functor from the finitely presented members of~$\Bool_P$ to~$\cS$.
Since every member of~$\Bool_P$ is a small directed colimit of a direct system of finitely presented objects, it follows, using \cite[Proposition~1.4.2]{Larder}, that this functor can be uniquely extended, up to isomorphism, to a functor from~$\Bool_P$ to~$\cS$ that preserves all small directed colimits.
This functor will also be denoted by ${}_{-}\otimes\vec{S}$.
For a $P$-scaled Boolean algebra~$\bA$, the object~$\bA\otimes\vec{S}$ will be called a \emph{condensate} of~$\vec{S}$.

In the particular case where~$\gf$ is a normal morphism, $\gf\otimes\vec{S}$ is a directed colimit of projection morphisms (i.e., canonical morphisms $X\times Y\to X$).
Now in all the cases we will be interested in, $\cS$ will be a category of models of first-order languages, so projection morphisms are surjective, thus so are their directed colimits.
Hence, in all those cases, if~$\gf$ is a \emph{normal} morphism, then $\gf\otimes\vec{S}$ is \emph{surjective}.

\section{A non-Cevian lattice with countably based differences}
\label{S:aleph2}

Throughout this section, we shall consider $P$-scaled Boolean algebras with
 \[
 P=\Pow[3]=\set{\es,1,2,3,12,13,23,123}
 \]
(cf.~\eqref{Eq:Pow3}).
Since~$P$ has exactly three \jirr\ elements (viz.~$1$, $2$, $3$), it follows from Gillibert and Wehrung \cite[Proposition~4.2]{GilWeh2011} that the relation denoted there by $(\aleph_2,{<}\aleph_0)\leadsto P$ holds.
This means that for every mapping $F\colon\Pow(\go_2)\to[\go_2]^{<\go}$, there exists a one-to-one map $f\colon P\hookrightarrow\go_2$ such that $F(f[P\dnw x))\cap f[P\dnw y]\subseteq f[P\dnw x]$ whenever $x\leq y$ in~$P$.

Now define~$X$ as the poset denoted by~$P\seq{\go_2}$ in the proof of \cite[Lemma~3.5.5]{Larder}, together with the canonical isotone map $\partial\colon X\to P$.
This means that
 \[
 X=\setm{(a,u)}{a\in P\,,\ u\colon U\to\go_2\text{ with }a=\bigvee U}
 \]
with componentwise ordering ($\leq$ on the first component, extension ordering on the second one), and $\partial(a,u)=a$ whenever $(a,u)\in X$.
It follows from (the proof of) \cite[Lemma~3.5.5]{Larder} that $X$ is lower finite and that furthermore, it is, together with the map~$\partial$, a \emph{principal $\go$-lifter} of~$P$ (cf. Definition~\ref{D:goLift}).

We apply the Armature Lemma to the following data:

\begin{itemize}
\item
$\cS$ is the category of all distributive lattices with zero with $0$-lattice homomorphisms;

\item
$\cA$ is the subcategory of~$\cS$ whose objects are the completely normal members of~$\cS$ with countably based differences, and whose morphisms are the closed $0$-lattice homomorphisms;

\item
$\Phi$ is the inclusion functor from~$\cA$ into~$\cS$.

\item
$\vec{S}\eqdef\vec{\bA}=\Idc\vec{A}$, where~$\vec{A}$ is the diagram introduced in Section~\ref{S:vecA}.
\end{itemize}

\begin{lemma}\label{L:cAhasdircolim}
$\cA$ is a subcategory of~$\cS$, closed under all small directed colimits and finite products.
\end{lemma}

\begin{proof}
The statement about finite products is straightforward.
Now let
 \[
 \vecm{D,\gd_i}{i\in I}=\varinjlim\vec{D}
 \]
within the category of all distributive lattices with zero and $0$-lattice homomorphisms,
where $\vec{D}=\vecm{D_i\,,\gd_i^j}{i\leq j\text{ in }I}$ is a direct system in~$\cA$.
Hence,
 \begin{gather}
 D=\bigcup\vecm{\gd_i[D_i]}{i\in I}\ \text{(directed union)}\,;
 \label{Eq:Aasunion}\\
 \text{For all }i\in I\text{ and all }x,y\in D_i\,,
 \ \gd_i(x)=\gd_i(y)\Rightarrow(\exists j\geq i)
 \pI{\gd_i^j(x)=\gd_i^j(y)}\,.\label{Eq:Aiasunion}
 \end{gather}
It is straightforward to verify, using~\eqref{Eq:Aasunion} and~\eqref{Eq:Aiasunion}, that~$D$ is a completely normal distributive lattice with zero (for this we do not need the assumption that the~$\gd_i^j$ are closed maps).

Let $i\in I$, let $x,x'\in D_i$ and $\ol{y}\in D$ such that $\gd_i(x)\leq\gd_i(x')\vee\ol{y}$.
By~\eqref{Eq:Aasunion}, there are $j\in I$ and $y\in D_j$ such that $\ol{y}=\gd_j(y)$; since~$I$ is directed, we may assume that $j\geq i$.
By~\eqref{Eq:Aiasunion}, there exists $k\geq j$ such that $\gd_i^k(x)\leq\gd_i^k(x')\vee\gd_j^k(y)$.
Since the map~$\gd_i^k$ is closed, there is $z\in D_i$ such that $x\leq x'\vee z$ and $\gd_i^k(z)\leq\gd_j^k(y)$.
Now the latter inequality implies that $\gd_i(z)\leq\gd_j(y)=\ol{y}$, thus completing the proof that the map~$\gd_i$ is closed.

Now let $\ol{x},\ol{y}\in D$.
We claim that $\ol{x}\ominus_D\ol{y}$ has a countable coinitial subset.
By~\eqref{Eq:Aasunion}, there are $i\in I$ and $x,y\in D_i$ such that $\ol{x}=\gd_i(x)$ and $\ol{y}=\gd_i(y)$.
Since~$D_i$ has countably based differences, $x\ominus_{D_i}y$ has a countable coinitial subset.
Since, by the paragraph above, $\gd_i$ is closed and by Lemma~\ref{L:Closed2CBD}, it follows that  $\ol{x}\ominus_D\ol{y}$ has a countable coinitial subset.
Therefore, $D$ has countably based differences.
\end{proof}

By Lemma~\ref{L:cAhasdircolim}, $\xF(X)\otimes\vec{\bA}$ (cf. Section~\ref{S:CrashCondens}) denotes the same object in~$\cA$ as in~$\cS$.
Denote it by~$\bB$.

\begin{theorem}\label{T:Main}
The structure~$\bB$ is a non-Cevian completely normal distributive lattice with zero and countably based differences.
It has cardinality~$\aleph_2$\,.
\end{theorem}

\begin{proof}
Since~$\bB$ is an object of~$\cA$, it is a completely normal distributive lattice with zero and countably based differences.
Since~$X$ has cardinality~$\aleph_2$\,, so does~$\xF(X)$, and so~$\xF(X)$ is the directed colimit of a diagram, indexed by a set of cardinality~$\aleph_2$\,, of finitely presented $P$-scaled Boolean algebras; since all~$\bA_p$ are countable, it follows that $\bB=\xF(X)\otimes\vec{\bA}$ has cardinality at most~$\aleph_2$\,.

We claim that~$\bB$ has cardinality exactly~$\aleph_2$\,.
Indeed, for every $\xi<\go_2$\,, denote by~$u_{\xi}$ the constant function on the singleton~$\set{123}$ with value~$\xi$.
The pair $v_{\xi}\eqdef(123,u_{\xi})$ belongs to~$X$ with $p\leq 123=\partial v_{\xi}$ whenever $p\in P$; thus $\tilde{v}_{\xi}\in\rF(X)^{(p)}$.
Hence, the Boolean subalgebra $V_{\xi}\eqdef\set{0,\tilde{v}_{\xi},\neg\tilde{v}_{\xi},1}$ of~$\rF(X)$, endowed with the ideals
 \[
 V_{\xi}^{(p)}\eqdef\begin{cases}
 \set{0,\tilde{v}_{\xi}}\,,&\text{if }p\neq\es\,,\\
 V_{\xi}\,,&\text{if }p=\es\,,
 \end{cases}
 \qquad\text{for }p\in P\,,
 \]
defines a finitely presented $P$-scaled Boolean algebra~$\bV_{\xi}$\,, and the inclusion map from~$\bV_{\xi}$ into~$\rF(X)$ defines a morphism $\bV_{\xi}\to\xF(X)$ in~$\Bool_P$\,, which in turns yields a morphism $e_{\xi}\colon\bV_{\xi}\otimes\vec{\bA}\to\xF(X)\otimes\vec{\bA}$ in~$\cA$.
Now pick any $\bu\in\bA_{123}\setminus\set{0}$.
Using the canonical isomorphisms $\bV_{\xi}\otimes\vec{\bA}\cong\bA_{123}\times\bA_{\es}\cong\bA_{123}$\,, it can be verified that the elements $e_{\xi}(\bu)$, for $\xi<\go_2$\,, are pairwise distinct.
(Think of~$e_{\xi}(\bu)$ as the constant map with value~$\bu$ on the clopen subset of the Stone space of~$\xF(X)$ associated to~$v_{\xi}$\,.)

Finally, towards a contradiction, we shall suppose that~$\bB$ has a Cevian operation~$\sd$.
Set
 \[
 X_{(k)}\eqdef\setm{x\in X}
 {\partial x\text{ has height }k\text{ within }P}\,,
 \]
for every nonnegative integer~$k$.
In particular, $X_{(k)}$ is nonempty if{f} $k\in\set{0,1,2,3}$, so~$X$ is the disjoint union of~$X_{(0)}$\,, $X_{(1)}$\,, $X_{(2)}$\,, $X_{(3)}$.
Further, $X_{(1)}=\partial^{-1}\set{1,2,3}$.

For each $x\in X$, the map $\rho_x\eqdef\pi_x^X\otimes\vec{\bA}$ is a surjective lattice homomorphism from~$\bB$ onto $\bA_{\partial x}$ (cf. 2.6.7, 3.1.2, and~3.1.3 in~\cite{Larder}).
In particular, if $x\in X_{(1)}$\,, then $\partial x\in\set{1,2,3}$, thus $\bA_{\partial x}\cong\set{0,1}$, and we may pick $\bb_x\in\bB$ such that $\rho_x(\bb_x)=1$.
For those~$x$, $\bB_x\eqdef\set{0,\bb_x}$ is a $0$-sublattice of~$\bB$.

Now for each $x\in X_{(2)}\cup X_{(3)}$\,, it follows from the lower finiteness of~$X$ that the $0$-sublattice~$\bB_x$ of~$\bB$ generated by all elements of~$\bB$ of the form either~$\bb_u$ or~$\bb_u\sd\bb_v$\,, where $u,v\in X_{(1)}\dnw x$, is finite%
\footnote{In the original statement of the Armature Lemma, we need~$\bB_{\bx}$ to be defined whenever~$\bx$ is a certain kind of \emph{ideal} of~$X$.
However, since~$(X,\partial)$ is a principal lifter of~$P$, it suffices here to consider the case where~$\bx$ is a principal ideal, which is then identified to its largest element.}
.
For any $x\in X$, $\bB_x$ is thus a finite $0$-sublattice of~$\bB$.
Denote by $\gf_x\colon\bB_x\hookrightarrow\bB$ the inclusion map, and by $\gf_x^y\colon\bB_x\hookrightarrow\bB_y$ the inclusion map in case $x\leq y$.
Hence $\vec{\bB}\eqdef\vecm{\bB_x,\gf_x^y}{x\leq y\text{ in }X}$ is an $X$-indexed commutative diagram in~$\cS$.

Since all~$\bB_x$ are finite, they are finitely presented within~$\cS$, thus we can apply the Armature Lemma~\cite[Lemma~3.2.2]{Larder} to those data, with the~$\bB_x$ in place of the required~$S_x$ and the identity of~$\bB$ in place of~$\chi$.
We get an isotone section $\gs\colon P\hookrightarrow X$ of~$\partial$ such that the family
 \[
 \vec{\chi}=
 \vecm{\chi_p}{p\in P}
 \eqdef\vecm{\rho_{\gs(p)}\res_{\bB_{\gs(p)}}}{p\in P}
 \]
is a natural transformation from $\vec{\bB}\gs\eqdef\vecm{\bB_{\gs(p)},\gf_{\gs(p)}^{\gs(q)}}{p\leq q\text{ in }P}$ to~$\vec{\bA}$.
This means that for all $p\leq q$ in~$P$, the square represented in Figure~\ref{Fig:vecchi} is commutative.

\begin{figure}[htb]
\begin{tikzcd}
\centering
\bB_{\gs(q)}\ar[r,"\chi_q"] & \bA_q\\
\bB_{\gs(p)}\ar[u,"\gf_{\gs(p)}^{\gs(q)}"]\ar[r,"\chi_p"] &
\bA_p\ar[u,"\bga_p^q" ']
\end{tikzcd}
\caption{The natural transformation~$\vec{\chi}$}\label{Fig:vecchi}
\end{figure}

For each $p\in\set{1,2,3}$, $\gs(p)\in X_{(1)}$ thus $\bB_{\gs(p)}=\set{0,\bb_{\gs(p)}}$ and
 \[
 \chi_p(\bb_{\gs(p)})=\rho_{\gs(p)}(\bb_{\gs(p)})=1\,,
 \]
so, using the commutativity of the diagram of Figure~\ref{Fig:vecchi} with $q\eqdef 123$,
 \begin{equation}\label{Eq:chibgsp2ap}
 \chi_{123}(\bb_{\gs(p)})=
 \pI{\chi_{123}\circ\gf_{\gs(p)}^{\gs(123)}}(\bb_{\gs(p)})=
 (\bga_{p}^{123}\circ\chi_p)(\bb_{\gs(p)})=\bga_{p}^{123}(1)
 =\ba_p
 \end{equation}
(we defined the~$\ba_i$ in~\eqref{Eq:ba1ba2ba3}).

Now we set
 \[
 \bc_{ij}\eqdef\chi_{123}(\bb_{\gs(i)}\sd\bb_{\gs(j)})\,,\quad
 \text{for all distinct }i,j\in[3]\,.
 \]
Since~$\gs(i)$ and~$\gs(j)$ both lie below~$\gs(ij)$, $\bb_{\gs(i)}\sd\bb_{\gs(j)}$ belongs to~$\bB_{\gs(ij)}$\,, thus
 \[
 \bc_{ij}=
 \pI{\chi_{123}\circ\gf_{\gs(ij)}^{\gs(123)}}(\bb_{\gs(i)}\sd\bb_{\gs(j)})
 =\pI{\bga_{ij}^{123}\circ\chi_{ij}}(\bb_{\gs(i)}\sd\bb_{\gs(j)})
 \]
belongs to the range of~$\bga_{ij}^{123}$.

Since~$\sd$ is a Cevian operation, the inequality $\bb_{\gs(i)}\leq\bb_{\gs(j)}\vee(\bb_{\gs(i)}\sd\bb_{\gs(j)})$ holds, thus, applying the homomorphism~$\chi_{123}$ and using~\eqref{Eq:chibgsp2ap}, we get
 \[
 \ba_i\leq\ba_j\vee\bc_{ij}\,.
 \]
Similarly, the equation
 \[
 (\bb_{\gs(i)}\sd\bb_{\gs(j)})\wedge(\bb_{\gs(j)}\sd\bb_{\gs(i)})=0
 \]
holds, thus, applying~$\chi_{123}$\,, we get
 \[
 \bc_{ij}\wedge\bc_{ji}=0\,.
 \]
Finally, the inequalities
\[
 (\bb_{\gs(1)}\sd\bb_{\gs(2)})\wedge(\bb_{\gs(2)}\sd\bb_{\gs(3)})
 \leq\bb_{\gs(1)}\sd\bb_{\gs(3)}\leq
 (\bb_{\gs(1)}\sd\bb_{\gs(2)})\vee(\bb_{\gs(2)}\sd\bb_{\gs(3)})
 \]
hold (use Lemma~\ref{L:Cev4}), thus, applying~$\chi_{123}$\,, we get
 \[
 \bc_{12}\wedge\bc_{23}\leq\bc_{13}\leq\bc_{12}\vee\bc_{23}\,.
 \]
We have thus proved that the elements~$\bc_{ij}$\,, for $i\neq j$ in~$[3]$, satisfy Conditions~\eqref{cijinrng}--\eqref{cijCeva} in the statement of Lemma~\ref{L:IdcANonRepr}; a contradiction.
\end{proof}

We obtain the following object (as opposed to diagram) version of Corollary~\ref{C:NotCongIdcvecG}.

\begin{corollary}\label{C:Main}
There exists a non-Cevian bounded completely normal distributive lattice with countably based differences, of cardinality~$\aleph_2$\,.
Hence, $\bB'$ is not a homomorphic image of~$\Csc{G}$ for any \lgrp~$G$ or of~$\Idc{G}$ for any representable \lgrp~$G$.
\end{corollary}

\begin{proof}
Let~$\bB'$ be obtained by adding a new top element to~$\bB$.
Since~$\bB$ is an ideal of~$\bB'$ and~$\bB$ is not Cevian, neither is~$\bB'$ (cf. Lemma~\ref{L:ProdHom}).

The last part of the statement of Corollary~\ref{C:Main} follows immediately, using Lem\-ma~\ref{L:ProdHom}, from Propositions~\ref{P:CevaSclgrp} and~\ref{P:CevaIdRepr}.
\end{proof}

\begin{remark}\label{Rk:Gamp}
A blunt application of \cite[Lemma~3.4.2]{Larder} (called there~CLL) to the non-lifting result obtained in Corollary~\ref{C:NotCongIdcvecG} would have yielded, in the statement of Corollary~\ref{C:Main}, a counterexample of cardinality~$\aleph_3$ (as opposed to~$\aleph_2$).
In order to get around that difficulty, we are applying here the Armature Lemma to the result of Lemma~\ref{L:IdcANonRepr}, which can be viewed as a ``local'' version of Corollary~\ref{C:NotCongIdcvecG}.
This technique was first put to full use in Gillibert~\cite{Gamps}.
It was instrumental in the proof, in Gillibert deep paper~\cite{PCPL}, that \emph{the congruence class of any finitely generated lattice variety~$\cV$ determines the pair consisting of~$\cV$ and its dual variety}.
\end{remark}

\begin{remark}\label{Rk:asuivre}
In the sequel~\cite{NonElt} to the present paper, we investigate an apparently innocuous extension of the construction $\bA\otimes\vec{S}$, denoted there by~$\bA\otimes^{\gl}_{\Phi}\vec{S}$.
This construction, applied to $\vec{S}\eqdef\vec{A}$ (the diagram introduced in Section~\ref{S:vecA}) yields for example that \emph{for any infinite cardinal~$\gl$, the class of principal ideal lattices of all Abelian \lgrp{s} with order-unit is not closed under $\scL_{\infty\gl}$-elementary equivalence}.
\end{remark}

\begin{problem}
Let~$D$ be a Cevian lattice with zero and with countably based differences.
Is there an Abelian \lgrp~$G$ such that $D\cong\Idc{G}$?
\end{problem}

The counterexample given in Remark~\ref{Rk:CevaIdRepr} shows that ``Cevian'' alone is not sufficient to get representability as~$\Idc{G}$, while Corollary~\ref{C:Main} shows that ``countably based differences'' alone is also not sufficient.
On the other hand, both ``Cevian'' and ``countably based differences'' are preserved under retracts, while it is not known whether any retract of a lattice of the form~$\Idc{G}$, for~$G$ an Abelian \lgrp, has this form (cf. \cite[Problem~2]{MV1}).
This would rather suggest a negative answer to the Problem above.


\providecommand{\noopsort}[1]{}\def\cprime{$'$}
  \def\polhk#1{\setbox0=\hbox{#1}{\ooalign{\hidewidth
  \lower1.5ex\hbox{`}\hidewidth\crcr\unhbox0}}}
  \providecommand{\bysame}{\leavevmode\hbox to3em{\hrulefill}\thinspace}
\providecommand{\MR}{\relax\ifhmode\unskip\space\fi MR }
\providecommand{\MRhref}[2]{%
  \href{http://www.ams.org/mathscinet-getitem?mr=#1}{#2}
}
\providecommand{\href}[2]{#2}

\end{document}